\documentclass{amsart}

\usepackage{amssymb}
\usepackage{graphicx}
\usepackage{url}

\usepackage{tikz}
\usetikzlibrary{matrix,decorations.pathreplacing, calc, positioning,fit}
\usepackage{pgfplots}
\pgfplotsset{compat=1.15}
\usepackage{mathrsfs}
\usetikzlibrary{arrows}
\usepackage{pst-all}
\usepackage{pstricks-add}

\newtheorem{theorem}{Theorem}[section]
\newtheorem{lemma}[theorem]{Lemma}
\newtheorem{conjecture}[theorem]{Conjecture}
\newtheorem{claim}[theorem]{Claim}

\newtheorem{observation}[theorem]{Observation}

\theoremstyle{definition}

\theoremstyle{remark}

\numberwithin{equation}{section}

\newcommand*{\myproofname}{Proof}
\newenvironment{myproof}[1][\myproofname]{\begin{proof}[#1]\renewcommand*{\qedsymbol}{\(\blacksquare\)}}{\end{proof}}

\DeclareMathOperator{\acosh}{acosh}
\DeclareMathOperator{\acot}{acot}
\DeclareMathOperator{\atanh}{atanh}
\DeclareMathOperator{\sys}{sys}
\DeclareMathOperator{\diam}{diam}
\DeclareMathOperator{\cat}{cat}

\definecolor{xdxdff}{rgb}{0.49019607843137253,0.49019607843137253,1}
\definecolor{uququq}{rgb}{0.25098039215686274,0.25098039215686274,0.25098039215686274}
\definecolor{qqqqff}{rgb}{0,0,1}
\definecolor{aqaqaq}{rgb}{0.6274509803921569,0.6274509803921569,0.6274509803921569}

\begin{document}

\title{Cops and Robber on Hyperbolic Manifolds}

\author{Vesna Ir\v si\v c}
\address{University of Ljubljana}
\email{vesna.irsic@fmf.uni-lj.si}
\thanks{V.I.~was supported by the Slovenian Research and Innovation Agency (ARIS) under the grant Z1-50003, and by the European Union (ERC, KARST, 101071836).}

\author{Bojan Mohar}
\address{Simon Fraser University}
\email{mohar@sfu.ca}
\thanks{B.M. was supported in part by an NSERC Discovery Grant R611450 (Canada),
by the ERC Synergy grant (European Union, ERC, KARST, project number
101071836),
and by the Research Project N1-0218 of ARIS (Slovenia). On
leave from:
FMF, Department of Mathematics, University of Ljubljana, Ljubljana,
Slovenia.}

\author{Alexandra Wesolek}
\address{Technische Universität Berlin}
\email{alexandrawesolek@gmail.com}
\thanks{A.W. was supported by the Vanier Canada Scholarship program and by the Deutsche Forschungsgemeinschaft (DFG, German Research Foundation) under Germany's Excellence Strategy – The Berlin Mathematics Research Center MATH+ (EXC-2046/1, project ID:
390685689).}


\date{}

\begin{abstract}
The Cops and Robber game on geodesic spaces is a pursuit-evasion game with discrete steps which captures the behavior of the game played on graphs, as well as that of continuous pursuit-evasion games. One of the outstanding open problems about the game on graphs is to determine which graphs embeddable in a surface of genus $g$ have largest cop number. It is known that the cop number of genus $g$ graphs is $O(g)$ and that there are examples whose cop number is $\tilde\Omega(\sqrt{g}\,)$. The same phenomenon occurs when the game is played on geodesic surfaces. 

In this paper we obtain a surprising result about the game on a surface with constant curvature. It is shown that two cops have a strategy to come arbitrarily close to the robber, independently of the genus. We also discuss upper bounds on the number of cops needed to catch the robber. 

Our results generalize to higher-dimensional hyperbolic manifolds.

\end{abstract}

\maketitle

\section{Introduction}
\label{sec:intro}

The Cops and Robber game is a pursuit-evasion game. The game is commonly played on graphs~\cite{AIGNER19841,bollobas2013cops,MR4454844,BoNo11,bradshaw2020proof,FRANKL1987301,joret2008cops,luczak2010chasing}, and as a new variant on geodesic spaces~\cite{MR4517712,Mo21,Mo22}. The players in the Cops and Robber game are the robber $r$ and $k$ cops $c_1,\dots,c_k$. On graphs the players occupy vertices while on geodesic spaces the players occupy points in space. The game is played in rounds. The robber chooses initial positions for the players $r^0, c_1^0,\dots,c_k^0$ and in each round of the game the players can move to a new position\footnote{The rules of the Cops and Robber game on graphs we define here are slightly different to standard rules, but they do not affect the outcome of the game on connected graphs.}. Each round of the game has two turns, the first one for the robber and the second one for the cops.
In particular, each of the cops can make a step at the cops' turn. When the game is played on a graph, the players are allowed to move to an adjacent vertex at their turn. When the game is played on a geodesic space $X$, the robber chooses an agility function $\tau: \mathbb{N}\to \mathbb{R_+}$ at the beginning of the game such that $\sum_{n \geq 1} \tau(n)=\infty$. In the $n$-th round, each player makes a step of length at most $\tau(n)$ (at their turn). The position of a player $x$ after round $n$ is denoted as $x^n$. In the following we give the robber the pronoun he, him, while the cops have the pronoun she, her. We say the cops \emph{catch} the robber if at some point in the game the cop $c_i$ occupies the same position as the robber $r$. If the robber is not caught, we say the robber \emph{escapes}.

The \emph{cop number $c(G)$} of a graph $G$ is the minimum number of cops that can catch the robber (regardless of the robber's strategy and initial positions). For a geodesic space $X$ we denote by $c_0(X)$ the \emph{cop catch number}, which is the minimum number of cops that can catch the robber. Further, if the game is played on a geodesic space, we say that the \emph{cops win the game} if 
\begin{align}
\label{eq:copwin}
    \inf_{n,i} d(c_i^n,r^n)=0.
\end{align}
The minimum number of cops that win the game on a geodesic space $X$ is the \emph{cop win number} $c(X)$ of $X$.
Note that in~\cite{Mo21, Mo22, MR4517712} the cop win number is referred as the cop number, and the cop catch number is called the strong cop number. Informally speaking, the cop win number is the minimum number of cops needed to capture the robber when the players are represented by infinitesimally small disks in the metric space.

The Cops and Robber game on geodesic spaces is tightly related to continuous and discrete pursuit-evasion games on metric spaces. In continuous pursuit-evasion games the players make decisions at every point in the time interval $[0,\infty)$. For example, the Lion and Man game as introduced by Rado (see Littlewood's Miscellany~\cite{Li86}) is a two-player game where each player is represented by a point in a metric space, and they can both move with the same maximal speed in any direction. More precisely, the players' paths are 1-Lipschitz maps and a strategy of a player consists of choosing a path $f(g)$ for each of the possible opponents paths $g$ such that if $g$ and $g'$ agree on a time interval $[0,t]$, then $f(g)$ and $f(g')$ agree on $[0,t]$. The man is trying to evade the capture by the lion, while the lion is aiming to occupy the same point as the man. Besicovitch showed that the man can escape the lion when the game is played on a disk \cite{Li86}. Croft studied a variation of this game with multiple pursuers on higher dimensional balls~\cite{Croft64} and Satimov and Kushkarov studied the game on the sphere~\cite{SaKu00}. However, Bollobás, Leader and Walters~\cite{BoLeWa12} showed that in some geodesic spaces, the lion and man game is not well-defined, in the sense that both the lion and the man have a winning strategy.

The Discrete Lion and Man game is the Cops and Robber game where the agility function is constant, i.e. $\tau \equiv K$ for some constant $K$. It was shown that in the Discrete Lion and Man game the lion can catch the man on a disk, more generally, the lion can catch the man on any compact $\cat(0)$-space~\cite{beveridge2015two,MR3915474}.
While the Cops and Robber game is a discrete pursuit-evasion game, it captures the properties of the continuous game, in the sense that in the Cops and Robber game one cop cannot catch the robber when the game is played on a disk~\cite{MR4517712}. Further, the Cops and Robber game is well defined, as Mohar showed that either the cops or the robber have a winning strategy~\cite{Mo21}.

One of the first results about cop numbers given by Aigner and Fromme states that planar graphs have cop number at most $3$~\cite{AIGNER19841}. 

\begin{theorem}[\cite{AIGNER19841}]
\label{thm:copnumberplanar}
    The cop number of every planar graph is at most three.
\end{theorem}

An important tool for determining upper bounds for cop numbers is the Isometric Path Lemma, which we use to prove our main theorem and to show that the cop catch number for special surfaces of genus $g\geq 2$ is at most $6$ (see Theorem~\ref{thm:catch}).

\begin{lemma}[\cite{AIGNER19841,Mo22}]
\label{lem:guardingpath}
   Let $I$ be an isometric path starting at $A$ and ending at $B$. Then one cop $c$ can guard $I$ after spending time equal to the length of $I$ on the path to adjust himself.
\end{lemma}

The idea of the lemma is that the cop $c$ stays at the same distance to $A$ as the robber, unless the robber is too far away, in which case $c$ waits at $B$.

For graphs embeddable in a surface of genus $g$ it is known that the cop number is at most linear in $g$ and recent progress was made on improving the linear factor, see~\cite{bowler2019bounding,erde2021improved,schroder2001copnumber}. It is an outstanding open problem to determine which graphs embeddable in a surface of genus $g$ have largest cop number. There are graphs of genus $g$ with cop number at least $g^{\frac{1}{2}-o(1)}$; one such example are binomial random graphs ${\mathcal G}(n,p)$ with $p=\frac{2\log n}{n}$~\cite{bollobas2013cops,mohar2017notes}. The gap between the upper and lower bound is large and it is conjectured that the lower bound gives the right order of the cop number.

\begin{conjecture}[\cite{mohar2017notes,Mo22}]
\label{conj:Mohar}
   Let $S$ be a a graph of genus $g$. Then $c(S)=O(\sqrt{g})$.
\end{conjecture}
Conjecture~\ref{conj:Mohar} resembles the famous Meyniel's conjecture. 

\begin{conjecture}[\cite{FRANKL1987301}]
\label{conj:Meyniel}
   Let $G$ be a a graph of order $n$. Then $c(G)=O(\sqrt{n})$.
\end{conjecture}

This is a long-standing conjecture and the best upper bound is that the cop number of an $n$-vertex graph is at most $O\left(\frac{n}{2^{(1-o(1))\sqrt{\log^2(n)}}}\right)$, see~\cite{FrKrLo12, LuPe12, ScSu11}. Hence even a weaker version of the conjecture, called the Weak Meyniel's conjecture, which states that there exists some $\varepsilon>0$ such that $c(G)\leq O(n^{1-\varepsilon})$, is still open. Conjecture~\ref{conj:Mohar} for graphs of genus $g$ implies that Meyniel's conjecture holds for graphs with $O(n)$ edges as their genus is $O(n)$, which includes for example subcubic graphs. Hosseini, Mohar and Gonzalez Hermosillo de la Maza~\cite{HoMoGo21} showed that Meyniel's conjecture for subcubic graphs implies the Weak Meyniel's conjecture (with exponent $\frac{3}{4}$), thus
Conjecture~\ref{conj:Mohar} implies the Weak Meyniel's conjecture.

As for graphs, the cop win number for a surface of genus $g$ is at most linear in $g$~\cite{Mo21}. It was shown that for graphs of cop number at least $3$ there exists a surface $S$ of genus $g$ with $c(S)\geq c(G)$~\cite{Mo21}. Therefore there are compact surfaces of genus $g$ with cop win number at least $g^{\frac{1}{2}-o(1)}$. The following conjecture is a tough conjecture since it implies Conjecture~\ref{conj:Mohar}.

\begin{conjecture}[\cite{Mo22}]
\label{conj:Mohar2}
   Let $S$ be a geodesic surface of genus $g$. Then $c(S)=O(\sqrt{g})$.
\end{conjecture}

The main result of this paper shows that the upper bound from Conjecture~\ref{conj:Mohar2} can be significantly improved for surfaces of constant curvature. The bound is independent of the genus, which we find surprising. Surfaces of constant curvature are either spherical, Euclidean or hyperbolic. It follows from~\cite{MR4517712} that spherical and Euclidean surfaces, which are surfaces of genus $0$ and $1$, have cop win number $2$ and cop catch number $3$. Here we settle the case of hyperbolic surfaces which can have arbitrarily large genus, and extend it to higher-dimensional hyperbolic manifolds. The main result of the paper thus reads as follows.

\begin{theorem}
\label{thm:hyperbolicmanifold}
   If $M$ is a compact hyperbolic manifold, then $c(M)=2$.
\end{theorem}

Definitions and technical lemmas needed in the proofs are presented in Section~\ref{sec:preliminaries}.
The proof of the 2-dimensional version of Theorem~\ref{thm:hyperbolicmanifold} is given in Section \ref{sec:proof}.
In the proof we describe a novel strategy based on the covering space technique. One cop follows the robber while the other one assures that the robber needs to diverge from its route, enabling the first cop to take shortcuts.
The higher-dimensional version of Theorem~\ref{thm:hyperbolicmanifold}  is presented in Section~\ref{sec:generalizations}. 
Finally, the cop catch number is discussed in Section~\ref{sec:catch}.

\section{Preliminaries}
\label{sec:preliminaries}
In this section we present definitions and preliminary lemmas that we need to prove our results. We start with basic definitions which are needed for the Cops and Robber game.

Let $(X,d)$ be a metric space. For $x,y\in X$, an \emph{$(x,y)$-path} is a continuous map $\gamma: I\to X$ where $I=[0,1]$ is the unit interval on $\mathbb{R}$, $\gamma(0)=x$ and $\gamma(1)=y$. One can define the \emph{length} $\ell(\gamma)$ of the path $\gamma$ by taking the supremum over all finite sequences $0=t_0<t_1<t_2< \cdots < t_n=1$ of the values $\sum_{i=1}^n d(\gamma(t_{i-1}),\gamma(t_i))$. Clearly, the length of every $(x,y)$-path is at least $d(x,y)$.  The metric space $X$ is a \emph{geodesic space} if for every $x,y\in X$ there is an $(x,y)$-path whose length is equal to $d(x,y)$. An $(x,y)$-path $\gamma$ is \emph{isometric} if $\ell(\gamma) = d(x,y)$. Observe that for $0\le t < t' \le 1$ the subpath $\gamma|_{[t,t']}$ is also isometric. Therefore the set $\gamma(I) = \{\gamma(t)\mid t\in I\}$ is an isometric subset of $X$. With a slight abuse of terminology, we say that the image $\gamma(I)\subset X$ is an \emph{isometric path} in $X$. If there exists a unique isometric path between $x,y$ then we will denote it by $\overline{xy}\subset X$. A continuous map $\gamma: \mathbb{R}\to X$ is a \emph{geodesic} if it is locally isometric, i.e., for every $t\in \mathbb{R}$ there is an $\varepsilon>0$ such that the subpath $\gamma|_J$ on the interval $J = [t-\varepsilon,t+\varepsilon]$ is isometric.

\subsection{Hyperbolic plane and hyperbolic surfaces}
Since we consider the Cops and Robber game played on hyperbolic spaces, we properly define those and give examples for particular hyperbolic surfaces.

The $n$-dimensional hyperbolic space $\mathbb{H}^n$ is the unique $n$-dimensional simply connected complete Riemannian manifold with a constant negative sectional curvature equal to $-1$. Standard examples for hyperbolic space is the half space model and the Poincaré disk model. The $n$-dimensional \emph{Poincaré disk} $\mathcal{D}^n$ is the disk $\{(x_1, \ldots, x_n) \mid x_1^2 + \cdots + x_n^2 < 1\}$ equipped with the hyperbolic metric $4 \frac{dx_1^2 + \cdots + dx_n^2}{(1 - x_1^2 - \cdots - x_n^2)}$. For brevity, we write $\mathcal{D}=\mathcal{D}^2$. 

Well-known examples for surfaces which are not hyperbolic are the sphere and the flat torus. The flat torus arises from a tessellation of the Euclidean plane by translations of a parallelogram. Here, the Euclidean plane is the universal covering space of the standard torus. For a topological space $X$, a \emph{covering space} of $X$ is a topological space $C$ together with a continuous surjective map $p \colon C \to X$
such that for every $x \in X$, there exists an open neighbourhood $U$ of $x$, such that $p^{-1}(U)$ is a union of disjoint open sets in $C$, each of which is mapped homeomorphically onto $U$ by $p$ (i.e.\ it is a local homeomorphism).

Similarly as for the flat torus, hyperbolic surfaces arise from a tessellation of the hyperbolic plane into copies of some polygon (see Figure~\ref{fig:Sg'}), which is formalized by the following theorem.

\begin{theorem}[Killing-Hopf \cite{Hopf1926,Killing1891}]
\label{thm:constantcurvature}
Any compact Riemannian manifold of constant curvature $-1$ is isometric to $\mathbb{H}^n/\Gamma$ where $\Gamma$ is a group of isometries on $\mathbb{H}^n$ acting freely and properly discontinuously. 
\end{theorem}
Here, the curvature is the sectional curvature. For more information about curvatures, we refer the reader to~\cite{AbTo12}. 

A \emph{fundamental polygon} of a Riemannian surface $\mathbb{H}^2/\Gamma$ is a polygon in $\mathbb{H}^2$ which contains a
representative of each $\Gamma$-orbit and at most one representative of each $\Gamma$-orbit in its interior. We define by \emph{$P(k,\theta)$} the regular $k$-gon in the Poincaré disk $\mathcal{D}$ centred at $O=(0,0)$ with angle $\theta$ at the vertices. We denote its vertices by $v_1, \dots, v_{k}$ in counter-clockwise direction and let $a_i$ be the (oriented) edge from $v_i$ to $v_{i+1}$ and $a_i^{-1}$ be the reversed edge from $v_{i+1}$ to $v_i$ (we consider the indices modulo $k$).
We are going to consider three standard hyperbolic surfaces for $g\geq 2$, where one of them is non-orientable.
\begin{itemize}
    \item Let $S(g)$ be the orientable surface obtained from $P\left(4g,\frac{2\pi}{4g}\right)$ by identifying the (oriented) edges $a_{4i-3}$ with $a_{4i-1}^{-1}$ and $a_{4i-2}$ with $a_{4i}^{-1}$ for $i=1,\dots,g$. The surface $S(2)$ is depicted in Figure~\ref{fig:Sg'}.
    \item Let $S'(g)$
    be the orientable surface obtained from $P\left(4g+2,\frac{2\pi}{2g+1}\right)$ by identifying opposite (oriented) edges $a_{i}, a_{i+2g+1}^{-1}$ for $i=1,\dots, 2g$.
    \item Let $N(g)$ be the non-orientable surface obtained from $P\left(2g,\frac{2\pi}{2g}\right)$ by identifying the (oriented) edge $a_{2i}$ with $a_{2i+1}$ for $i=1,\dots g$.
\end{itemize}
 \begin{figure}[h!]
     \centering
     \resizebox{0.8\textwidth}{!}{
     \raisebox{-0.5\height}{
    \begin{tikzpicture}
    \node[anchor=south west,inner sep=0] at (0,0) {\includegraphics[width=6cm]{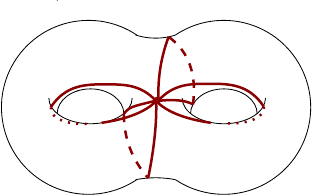}};
    \node at (3.3,3.5) {\large $a_1$};
    \node at (5,2.2) {\large $a_2$};
    \node at (3,0) {\large $a_5$};
    \node at (1,2.2) {\large $a_6$};
    \node at (0,-1) {\large $(a)$};
\end{tikzpicture}
}
     \hspace*{.2in}
     \raisebox{-0.5\height}{
\begin{tikzpicture}[scale=0.8]
    \node[anchor=south west,inner sep=0] at (-1,-1) {\includegraphics[width=6.5cm]{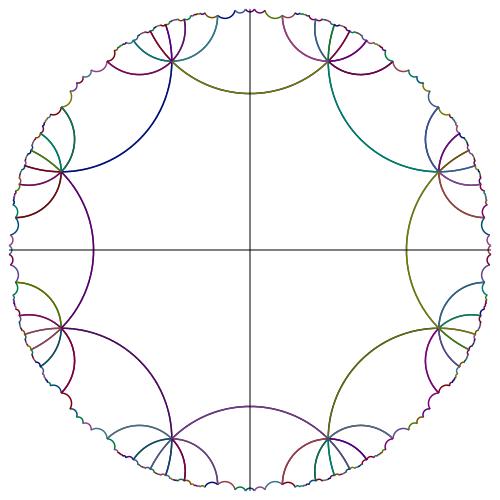}};
    \node at (5.1,5) {\large $a_1$};
    \node at (0.9,5) {\large $a_1^{-1}$};
    \node at (3,6) {\large $a_2$};
    \node at (0,3) {\large $a_2^{-1}$};
    \node at (0.9,1) {\large $a_5$};
    \node at (5.2,1) {\large $a_5^{-1}$};
    \node at (3,0) {\large $a_6$};
    \node at (6,3) {\large $a_6^{-1}$};
    \node at (5.5,4) {\large $v_1$};
    \node at (2,5.5) {\large $v_3$};
    \node at (4,5.5) {\large $v_2$};
    \node at (0.5,4) {\large $v_4$};
 \node at (5.5,2) {\large $v_8$};
    \node at (2,0.5) {\large $v_6$};
    \node at (4,0.5) {\large $v_7$};
    \node at (0.5,2) {\large $v_5$};   
     \node at (-1,0.2) {\large $(b)$};
\end{tikzpicture}
}
}
     \caption{Figure~(a) depicts the surface $S(g)$ and Figure~(b) its fundamental domain $P\left(4g,\frac{2\pi}{4g}\right)$ inside a partial tessellation of the Poincaré disk.
     }
           \label{fig:Sg'}
       \end{figure}
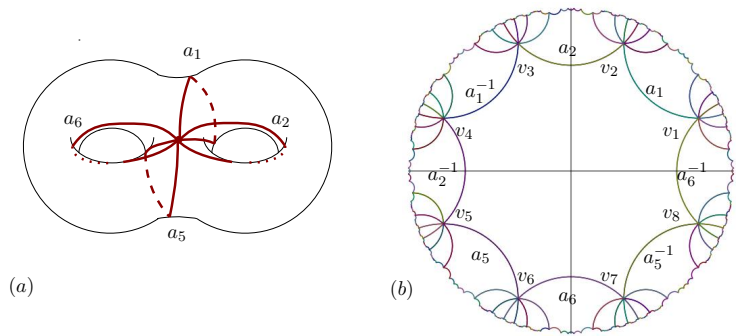
The \emph{diameter} of a compact surface $S$ is the maximum distance between points in $S$, that is $\diam(S) = \max\{d(x, y) \mid x, y \in S\}$. 
The \emph{systolic girth}, denoted as $\sys(S)$, is the length of a shortest noncontractible curve on the surface. A hyperbolic surface is locally isometric to the hyperbolic plane, which we will use frequently in our proofs.

\begin{lemma}
    \label{lem:ball-sys}
    If $S$ is a hyperbolic surface, $a \in S$ and $r < \frac{\sys(S)}{4}$, then the disk $B(a,r)$ is isometric to a hyperbolic disk.
\end{lemma}

\subsection{Trigonometric lemmas for the hyperbolic plane}
In the rest of this section, we present lemmas that will be used in the proof of Theorem~\ref{thm:mainthm}. Their proofs are omitted and are given in Appendix~\ref{ap:proofs}. For properties of the hyperbolic plane needed in the proofs of these lemmas we refer the reader to~\cite{MR0666074,slothers2023hyperbolic}.

First, consider Figure~\ref{fig:optimizing1} and let the robber's position be $X_1$ and the cop's position be $A$. The robber moves to $B$, while the cop moves as close to $B$ as possible in the same time. Our first lemma shows that if their starting positions would be closer together (at positions $X_2$ and $X_3$, say), then their end positions would be closer together.

\begin{lemma}
\label{lem:closertogether}
    Let $ABC$ be a hyperbolic triangle with a right angle at $C$. If $X_1\in \overline{AC}$, then 
    \begin{align*}
        d(A,B)-d(X_1,B)=\max_{X_2,X_3\in \overline{AX_1}}d(X_2,B)-d(X_3,B).
    \end{align*}
    Moreover,  $d(A,B)-d(X_1,B)=d(X_2,B)-d(X_3,B)$ if and only if $X_2= A$ and $X_3 = X_1$.
\end{lemma}

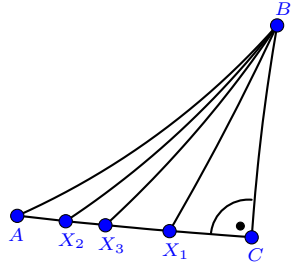
\begin{figure}
    \centering
\begin{tikzpicture}[line cap=round,line join=round,>=triangle 45,x=6cm,y=6cm]
\draw [shift={(-1.0945484369206189,1.6269683302302145)},line width=0.8pt]  plot[domain=5.131865244901581:5.5584103858721665,variable=\t]({1*1.6867312258737188*cos(\t r)+0*1.6867312258737188*sin(\t r)},{0*1.6867312258737188*cos(\t r)+1*1.6867312258737188*sin(\t r)});
\draw [shift={(-1.889244749094016,1.889753332799505)},line width=0.8pt]  plot[domain=5.4557916287934285:5.692030265211905,variable=\t]({1*2.4779857507270244*cos(\t r)+0*2.4779857507270244*sin(\t r)},{0*2.4779857507270244*cos(\t r)+1*2.4779857507270244*sin(\t r)});
\draw [shift={(-1.396481089915706,1.726809454163551)},line width=0.8pt]  plot[domain=5.298209924566212:5.621705071786334,variable=\t]({1*1.982934775901815*cos(\t r)+0*1.982934775901815*sin(\t r)},{0*1.982934775901815*cos(\t r)+1*1.982934775901815*sin(\t r)});
\draw [shift={(-4.998071046329796,2.9177597724824427)},line width=0.8pt]  plot[domain=5.75654804903177:5.846859810074327,variable=\t]({1*5.70035404813396*cos(\t r)+0*5.70035404813396*sin(\t r)},{0*5.70035404813396*cos(\t r)+1*5.70035404813396*sin(\t r)});
\draw [shift={(0.8263738222530633,10.680875632417418)},line width=0.8pt]  plot[domain=4.596439839762474:4.645348350774679,variable=\t]({1*10.666020718584496*cos(\t r)+0*10.666020718584496*sin(\t r)},{0*10.666020718584496*cos(\t r)+1*10.666020718584496*sin(\t r)});
\draw [shift={(4.624502042941939,-0.26417004806988514)},line width=0.8pt]  plot[domain=2.9698623677174503:3.0745520239797623,variable=\t]({1*4.5228094100317*cos(\t r)+0*4.5228094100317*sin(\t r)},{0*4.5228094100317*cos(\t r)+1*4.5228094100317*sin(\t r)});
\begin{scriptsize}
\draw [fill=qqqqff] (-0.407572866037683,0.08647269786135825) circle (2.5pt);
\draw[color=qqqqff] (-0.4100680595451122,0.04553425054000529) node {$A$};
\draw [fill=qqqqff] (-0.21215204102583265,0.06553475232437422) circle (2.5pt);
\draw[color=qqqqff] (-0.19805580687064303,0.02186622246506424) node {$X_3$};
\draw [fill=qqqqff] (0.16822063622937605,0.508721266190536) circle (2.5pt);
\draw[color=qqqqff] (0.1821195948078676,0.545075887634313) node {$B$};
\draw [fill=qqqqff] (0.11185258814678768,0.038814865099605154) circle (2.5pt);
\draw[color=qqqqff] (0.12099479351938675,0.00028873218786654) node {$C$};
\draw [fill=qqqqff] (-0.30014891839089974,0.07451223660235051) circle (2.5pt);
\draw[color=qqqqff] (-0.28669773990449265,0.031715326135492) node {$X_2$};
\draw [fill=qqqqff] (-0.07010261839926568,0.05259599925810576) circle (2.5pt);
\draw[color=qqqqff] (-0.05622871401648361,0.00807747732646538) node {$X_1$};
\draw[line width=0.8pt] (0.11185258814678768,0.038814865099605154)+(90:0.5cm) arc[start angle=85, end angle=175, radius=0.5cm];
\draw[fill=black] (0.08685258814678768,0.063814865099605154)  circle (1.5pt);
\end{scriptsize}
\end{tikzpicture}
  \caption{A hyperbolic right-angled triangle $ABC$ with the right angle at $C$.}

  \label{fig:optimizing1}
\end{figure}

The second lemma shows that $d(A,B)-d(X_3,B)$ would decrease if the length of the side $\overline{AC}$ would decrease, while $d(A,X_3)$ would remain the same. 

\begin{lemma}
\label{lem:shifttoright}
    Let $ABC$ be a hyperbolic triangle with a right angle at $C$. If $X_3\in \overline{AC}$, then 
    \begin{align*}
        d(A,B)-d(X_3,B)=\max_{\substack{X_1,X_2\in \overline{AC}\\d(X_1,X_2)=d(A,X_3)}}d(X_2,B)-d(X_1,B).
    \end{align*}
\end{lemma}

The next lemma helps to compare triangles where two sides have the same length.
\begin{lemma}
\label{lem:intuition}
    Suppose $ABC$, $A'B'C'$ are hyperbolic triangles, $d(A,B)=d(A',B')$, $d(B,C)=d(B',C')$, $\angle BCA>\angle B'C'A'$ and $\angle CAB<\frac{\pi}{2}$. Then it holds that $d(A,C)<d(A',C')$.
\end{lemma}
The Pythagorean theorem for the hyperbolic plane for a triangle $ABC$ with right angle at $C$ says that $\cosh(\overline{AB})=\cosh(\overline{AC})\cdot \cosh(\overline{BC}).$ In the following we observe that the increase of $\overline{BC}$ increases $\overline{AB}$ "uniformly". 
\begin{lemma}
\label{lem:mean-value}
    Suppose $t_2>t_1>0$ and $w>1$, then there exists $\eta>0$ such that for $x,y\in [t_1,t_2]$ with $x < y$ it holds that,
    \begin{align*}
        \acosh(w\cosh(x))-\acosh(w\cosh(y))\geq \eta \cdot (y-x).
    \end{align*}
\end{lemma}

\section{Hyperbolic surfaces}
\label{sec:proof}

In this section we prove that two cops can win the game on any compact hyperbolic surface. In fact, this is the most difficult step towards proving Theorem~\ref{thm:hyperbolicmanifold}.
\begin{theorem}
\label{thm:mainthm}
   If $S$ is a compact hyperbolic surface, then $c(S)=2$.
\end{theorem}

\subsection{The lower bound}
Let $s=\sys(S)$. To show that $c(S)>1$ we play the game with one cop $c$ and the robber $r$ and show that the robber has a strategy to stay at distance at least $\frac{s}{16}$ to the cop. The robber chooses the agility function  $\tau\equiv\frac{s}{16}$ and initial positions such that $d(c^0,r^0)>\frac{s}{16}$. Informally, the robber's strategy is to move in the direction opposite to the cop's position.

If $d(c^k,r^k)\geq \frac{3s}{16}$, then the robber's strategy is to stay at the same place, which means $r^{k+1}=r^k$. Then  $d(c^{k+1},r^{k+1})\geq d(c^k,r^{k+1})-\frac{s}{16}\geq\frac{s}{8}$. From Lemma~\ref{lem:ball-sys} it follows that if  $d(c^k,r^k)<\frac{3s}{16}$, then there exists a position $r^{k+1}$ such that $d(r^k,r^{k+1})=\frac{s}{16}$ and an isometric path between $r^{k+1}$ and $c^{k}$ contains $r^k$. This means that $d(c^k,r^{k+1})= d(c^k,r^k)+\frac{s}{8}>\frac{s}{4}$ and hence $d(c^{k+1},r^{k+1})>\frac{s}{8}$.

\subsection{The upper bound}
Instead of playing the Cops and Robber game on the hyperbolic surface $S$, we will consider the game on its universal cover, viewed as the Poincar\'e disk $\mathcal{D}$. For each time step $k$ we choose a representation $C_i^k, R^k$ of $c_i^k,r^k$ in $\mathcal{D}$ ($i=1,2$). 
 
 If the robber was only escaping from cop $c_1$, then by the above argument his strategy could be to stay roughly on the geodesic $g$ defined in $\mathcal{D}$ by the current positions of the robber $r$ and cop $c_1$. Since we have two cops available, we can force the robber to take a different strategy. The idea of the cops' strategy is that cop $c_1$ will chase the robber, while cop $c_2$ will force the robber to move away from $g$, allowing cop $c_1$ to move closer to~$r$.  

Let $D=\diam(S)$. The robber chooses an agility function $\tau$ and starting positions, $c_1^0,c_2^0,r^0$. 
Let $0=t_0<t_1<t_2<t_3<\cdots$ be a sequence of integers representing time steps such that $$\sum_{k=t_i}^{t_{i+1}-1} \tau(k) \geq 32D.$$

 Our goal is to show that for every $\varepsilon>0$ the cops $c_1,c_2$ have a strategy in which the cop $c_1$ can come $\varepsilon$-close to the robber $r$. This suffices to show that the cop win number is at most $2$ by the definition of infimum, see~\cite[Corollary 6]{Mo22}. We will attain our goal by showing that in between time step $t_i$ and $t_{i+1}$, the cop $c_1$ can decrease her distance to the robber by some small constant $\delta$.
 
 \begin{lemma}
 \label{prop:epsilon}
     For every $\varepsilon>0$, there exists $\delta>0$ and strategies for the cops $c_1,c_2$, such that if $d(c_1^{t_i}, r^{t_i})\geq \varepsilon$ and the cops follow their strategy then either $c_2$ captures the robber, or we have
 \begin{align*}
    d(c_1^{t_{i+1}}, r^{t_{i+1}})\leq d(c_1^{t_{i}}, r^{t_{i}})-\delta.
\end{align*}
 \end{lemma}
 
The rest of the section is devoted to the proof of Lemma~\ref{prop:epsilon}.
Let $R^{0},C_1^{0}$ be such that $d(R^0,C_1^0)=d(r^0,c_1^0)$. 

\subsubsection{The strategy of cop $c_2$}

We explain the strategy of the cop $c_2$ first. Since we play the game in the covering space $\mathcal{D}$, we have arbitrarily many copies of the position of $c_2$ that we can choose from~(see Figure~\ref{fig:copies}). 
At each time step $t_i$, the cop $c_2$ chooses a new representation $C_2^{t_i}$, and therefore $d(C_2^{t_i-1}, C_2^{t_i})$ in $\mathcal D$ is possibly greater than $\tau(t_i)$ but for the projection onto $S$ we maintain the condition that $d(c_2^{t_i-1}, c_2^{t_i})\leq \tau(t_i)$. 
 In between the time steps $t_i,t_{i+1}$ we choose the representation of $c_2$ such that it is coherent with the agility function, which means $d(C_2^{k-1},C_2^{k})\leq \tau(k)$ for $t_i< k < t_{i+1}$.  We consider the geodesic $g_0=g_0(t_i)$ through $R^{t_i},C_1^{t_i}$. We choose $C_2^{t_i}$ for cop $c_2$ such that it is close to the geodesic $g_0$ but sufficiently far from $R^{t_i}$, which we make more precise in the following. 

\begin{figure}
    \centering
\includegraphics[width=0.4\textwidth]{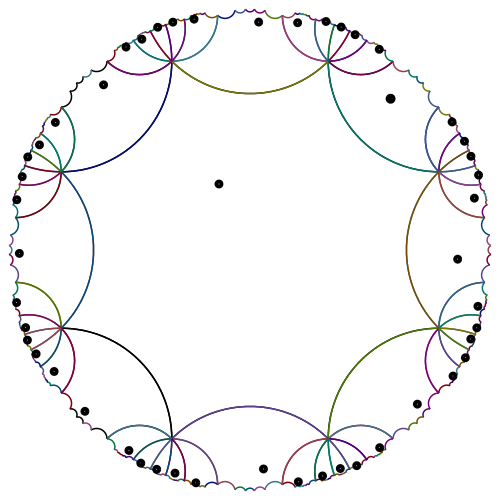}
    \caption{Copies of the fundamental polygon $P(4g,\frac{2\pi}{4g})$ in $\mathcal{D}$. Depicted is a point and one copy of this point in each of the depicted copies of the fundamental polygon.}
    \label{fig:copies}
\end{figure}

Let $P$ be the point on $g_0$ at distance $10D$ from $R^{t_i}$ that is further away from $C_1^{t_i}$. Note that there is a copy $C_2^{t_i}$ in the universal covering space which is at distance at most $D$ from $P$, this will be the starting position of $c_2$. We consider $h=o_{g_0}(C_2^{t_i})$, the orthogonal geodesic to $g_0$ through $C_2^{t_i}$. The strategy of $c_2$ between time steps $t_i$ and $t_{i+1}$ is to chase the orthogonal projection of the robber on $h$. Let $B,B'$ be points on $h$ at distance $8D$ from $g_0\cap h$.

\begin{lemma}
\label{lem:c2-guards}
    The cop $c_2$ can guard the path from $B$ to $B'$ on $h$.
\end{lemma}

\begin{proof}
First we show that the cop $c_2$ can reach $h \cap g_0$ much faster than the robber.
\begin{claim}
The point $h\cap g_0$ is at distance between $9D$ and $11D$ from $R^{t_i}$ and at distance at most $D$ from $C_2^{t_i}$.
     \label{cl:intersectionisfar}
\end{claim}
\begin{myproof}
 Note that the closest point from $C_2^{t_i}$ to $g_0$ is $h\cap g_0$ since $h$ and $g_0$ are orthogonal. Hence
$d(h\cap g_0,C_2^{t_i}) \leq d(P, C_2^{t_i}) \leq D$.

The closest point from $P$ to $h$ is also $h\cap g_0$ since  $P\in g_0$. In particular it holds that $d(h\cap g_0,P) \leq d(h\cap g_0, C_2^{t_i}) \leq D$. The distance from $R^{t_i}$ to $P$ is $10D$, so by the triangle inequality, robber's distance to $h\cap g_0$
is between $9D$ and $11D$. This proves the claim.
\end{myproof}

 Note that by Claim~\ref{cl:intersectionisfar} the distance from $R^{t_i}$ to $h$ is at least $9D$.  The cop $c_2$ can guard the path from $B$ to $B'$ on $h$, since her distance to $g_0\cap h$ is at most $D$, and hence her distance to 
 any point on $\overline{BB'}$ is at most $9D$, which is less than the distance from $R^{t_i}$ to that point. This means that $c_2$ can use the strategy from the proof of Lemma \ref{lem:guardingpath} to guard the isometric path $\overline{BB'}$. This completes the proof of Lemma~\ref{lem:c2-guards}.
 \end{proof}

 \subsubsection{The strategy of cop $c_1$}
 
 For $k\geq t_i$, let $g^k=o_{h}(C_1^k)$ (the orthogonal geodesic to $h$ through $C_1^k$) and let $B^k=g^k \cap h$, see Figure~\ref{fig:descriptionBk}. 

\begin{figure}[h!]
    \centering
\begin{tikzpicture}[ line cap=round,line join=round,>=triangle 45,x=5cm,y=5cm]
\clip(-0.5,-0.1) rectangle (0.6,0.6524311176419754);
\draw [shift={(-1.0945484369206189,1.6269683302302145)},line width=0.8pt]  plot[domain=5.131865244901581:5.5584103858721665,variable=\t]({1*1.6867312258737188*cos(\t r)+0*1.6867312258737188*sin(\t r)},{0*1.6867312258737188*cos(\t r)+1*1.6867312258737188*sin(\t r)});
\draw [shift={(4.6245020429414145,-0.2641700480698219)},line width=0.8pt]  plot[domain=2.8669299463564784:3.3021313904472964,variable=\t]({1*4.522809410031173*cos(\t r)+0*4.522809410031173*sin(\t r)},{0*4.522809410031173*cos(\t r)+1*4.522809410031173*sin(\t r)});
\draw [shift={(0.8263738222534689,10.68087563242184)},line width=0.8pt]  plot[domain=4.5416909151405:4.728655733763017,variable=\t]({1*10.666020718588936*cos(\t r)+0*10.666020718588936*sin(\t r)},{0*10.666020718588936*cos(\t r)+1*10.666020718588936*sin(\t r)});
\begin{scriptsize}
\draw [fill=qqqqff] (-0.407572866037683,0.08647269786135825) circle (2.5pt);
\draw[color=qqqqff] (-0.37897752174736895,0.025) node {$C_1^k$};
\draw [fill=qqqqff] (0.16822063622937605,0.508721266190536) circle (2.5pt);
\draw[color=qqqqff] (0.21211959480786763,0.5450758876343129) node {$B$};
\draw [fill=qqqqff] (0.11185258814678768,0.038814865099605154) circle (2.5pt);
\draw[color=qqqqff] (0.1609042557216434,-0.01) node {$B^k$};
\draw[line width=0.8pt] (0.11185258814678768,0.038814865099605154)+(90:0.5cm) arc[start angle=85, end angle=175, radius=0.5cm];
\draw[fill=black] (0.07185258814678768,0.078814865099605154)  circle (1.2pt);
\draw[color=black] (0.14863264232841333,0.17868923109440066) node {$h$};
\draw[color=black] (-0.05031925181422699,0.010080177970841217) node {$g^k$};
\draw [fill=qqqqff] (-0.18553271925832648,0.1229642031145643) circle (2.5pt);
\draw[color=qqqqff] (-0.12835759952978758,0.16) node {$R^k$};
\end{scriptsize}
\end{tikzpicture}
    \caption{A schematic picture of cop's position $c_1$ and robber's position $r$ at time step $k$. }
    \label{fig:descriptionBk}
\end{figure}
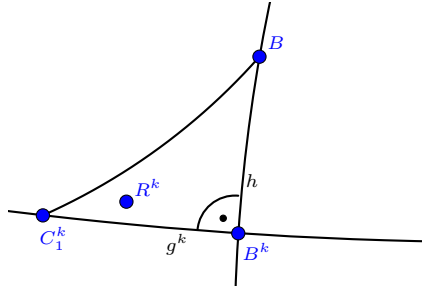

We will assign a strategy to cop $c_1$ that will force the robber $r$ to cross the isometric path $\overline{C_1^kB}$ (or the symmetric case $\overline{C_1^kB'}$) for some $k$. Note that the robber cannot cross $\overline{BB'}$ since it is guarded by cop $c_2$. Since in round $k$ the robber makes his turn before the cop, the robber crosses $\overline{C_1^kB}$ (or $\overline{C_1^kB'}$) in step $k+1$.

In the following the strategy of the cop $c_1$ will be such that the best strategy for the robber $r$ is to cross $\overline{C_1^kB}$ (or $\overline{C_1^kB'}$) somewhere close to $B$ (or $B'$). 
We say two points $X,Y$ \emph{are on opposite sides} of the geodesic $g$ if $X,Y$ are in distinct connected components of the Poincaré disk $\mathcal{D}\setminus g$. If two points $X,Y$ are not on opposite sides, we say they \emph{are on the same side}.

Suppose $R^k,R^{k+1}$ are on the same side as $B$ with respect to  $g^k$. We denote $x=o_{h}(R^{k+1}) \cap \overline{C_1^kB}$. 
Then the position $C_1^{k+1}$ can be chosen to be
\begin{itemize}
    \item[(a)] on $\overline{C_1^{k}B}$ with $d(C_1^{k+1},C_1^{k})=\tau(k+1)$ if $d(x,C_1^k)\geq \tau(k+1)$, or
    \item[(b)] on $o_{h}(R^{k+1})$, as close to the point $R^{k+1}$ as possible, otherwise. 
\end{itemize}
This strategy ensures the following condition:
\begin{equation} 
\label{stmt:coplower}
  \mbox{The robber's position $R^{k+1}$ and $B$ are on the same side of $g^{k+1}$.}
\end{equation}

\begin{observation}
    \label{obs:cop-moves-up}
    There is a strategy for the cop $c_1$, such that for every time step $t_i \leq k \leq t_{i+1}-1$, while the robber is contained in the triangle $BB^kC_1^k$ we can assume the robber's position $R^{k+1}$ is on the same side of the geodesic $g^k$ as $R^k$ by subdividing each step at most once.
\end{observation}

\begin{proof}
    By symmetry we can assume that the first step of the robber away from the geodesic $g_0$ is towards $B$. Now, suppose that the robber crosses $g^k$ by going from step $R^k$ to step $R^{k+1}$, see Figure~\ref{fig:reflection}. By splitting each such step into two substeps we may consider the first substep as above, initial and ending position on the same side as $B$, and the second substep with initial and ending position on the same side as $B'$. In the first substep the cop $c_1$ moves along $g^k$ towards the robber. In the second substep the cop imagines the robber to move to $\phi(R^{k+1})$, which is the reflection of the position $R^{k+1}$ along $g^k$, and determines his position $\phi(C_1^{k+1})$ by using strategy $(a)$ or $(b)$. She then reflects the position $\phi(C_1^{k+1})$ along $g^k$ to obtain $C_1^{k+1}$. After subdividing each of the robber's steps at most once, by symmetry we can assume the robber does not cross $g^k$, which means the position $R^{k+1}$ is on the same side of the geodesic $g^k$ as $B$.
\begin{figure}[htb]
\includegraphics[viewport=185 530 425 680,clip]{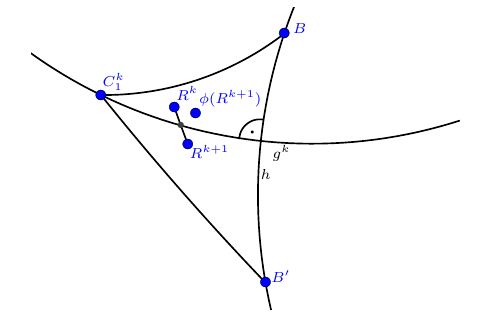}
    \caption{The point $\phi(R^{k+1})$ is the reflection point of $R^{k+1}$ over $g^k$.}
    \label{fig:reflection}
\end{figure}
\end{proof}

Recall that in round $k$ the robber makes a step before the cops make their step. We let $t$ be the first round at which the robber crosses $\overline{C_1^{t-1}B}$ (or $\overline{C_1^{t-1}B'}$). We will assume without loss of generality that $R^t=x\in \overline{C_1^{t-1}B}$ (or $\overline{C_1^{t-1}B'}$) by subdividing the step and letting the agility at time step $t$ be $d(x,R^{t-1})$. From step $t$ until step $t_{i+1}$, the strategy of cop $c_1$ is to simply follow the robber, while the strategy of cop $c_2$ is to stay at the same position. In particular, $R^t \in \overline{C_1^tB} = \overline{C_1^{t-1}B}$ (or $R^t \in \overline{C_1^tB'}$).

\subsubsection{Proof of Lemma~\ref{prop:epsilon}}

We start by showing that the robber crosses $\overline{C_1^tB}$.

\begin{lemma}
    If the robber does not get caught by the cops, then  $R^t \in \overline{C_1^{t} B}$ (or $R^t \in \overline{C_1^tB'}$) for some $t$ with $t_i<t < t_{i+1}$. In particular,
     $\sum_{k=t_i}^t \tau(k) < 32D$.
\end{lemma}
\begin{proof}
By Observation~\ref{obs:cop-moves-up} we can assume that at step $k$, the cop $c_1$ is moving closer to $B$ and $h$, and by the triangle inequality on the triangle defined by $C_1^{k-1},C_1^{k}$ and $g^k\cap \overline{C_1^{k-1}B}$, she decreases the total distance to $B$ and to $h$ by at least $\tau(k)$, i.e., 
$$
  d(C_1^k,B) + d(C_1^k,h)  \le d(C_1^{k-1},B) + d(C_1^{k-1},h) - \tau(k).
$$

By Claim~\ref{cl:intersectionisfar}, $d(C_1^{t_i},g_0\cap h)\leq 12D$ and by the triangle inequality it holds that $d(C_1^{t_i},B)\leq d(C_1^{t_i},g_0\cap h)+d(g_0 \cap h,B)\leq 12D+8D=20D$. Hence the cop either meets $B$ or $h$ after moving for time $32D$, thus catching the robber. 
\end{proof}

Lemma~\ref{prop:epsilon} follows directly from the following result.

\begin{lemma}
\label{lem:crossing}
  For each $\varepsilon>0$ there exists some $\delta>0$, such that if $d(R^{t_i},C_1^{t_i})\geq\varepsilon$, and at step $t$ $(t_i<t<t_{i+1})$ $R^t\in \overline{C_1^t B}$, then
\begin{align*}
    d(R^t,C_1^t)\leq d(R^{t_i},C_1^{t_i}) - \delta.
\end{align*}
\end{lemma}

The rest of this section is concerned with proving Lemma~\ref{lem:crossing}. The proof is technical and uses Lemmas~\ref{lem:closertogether}--\ref{lem:intuition}. 

\begin{proof}[Proof of Lemma~\ref{lem:crossing}]
Assume that the robber does not get caught by step $t_{i+1}$. Let us consider the right triangle on vertices $C_1^{t-1}$, $B^{t-1}$ and $B$; recall that $B^k=g^k \cap h$, see Figure~\ref{fig:descriptionBk}. We define a function $f \colon \overline{C_1^{t-1}B} \to \mathbb{R}$ with
$$f(x)=d(C_1^{t-1},x)-d(R^{t-1},x).$$
If $R^t=x$, then $d(C_1^t,R^t) \leq f(x)$. Note that by the triangle inequality it holds that $f(x)< d(C_1^{t-1},R^{t-1})$ since $x$, $C_1^{t-1}$ and $R^{t-1}$ are not on a common geodesic. In order to maximise his distance to cop $c_1$, we show that the best choice for the robber would be to move to $B$ in his last step (we ignore that the robber would be caught by $c_2$).

\begin{claim} 
Given fixed positions $C_1^{t-1}$ and $R^{t-1}$ such that $R^{t-1}$ is in the triangle defined by $C_1^{t-1}, B,B^{t-1}$, it holds that
    $$f(B)= \max_{x \in \overline{C_1^{t-1}B}} f(x).$$
    \label{cl:fmax}
\end{claim}
\begin{myproof}
Suppose $x\in \overline{C_1^{t-1}B}$ and $x\neq B$. We consider the triangle defined by $x,B,R^{t-1}$. Since $x$ is on a common geodesic with $C_1^{t-1}$ and $B$,
\begin{align*}
    f(B)& =d(C_1^{t-1},B)-d(R^{t-1},B) \\
    & =(d(C_1^{t-1},x)+d(x,B))-(d(R^{t-1},B)-d(R^{t-1},x)) -d(R^{t-1},x).
\end{align*}
By the triangle inequality, $d(x,B)\geq d(R^{t-1},B)-d(R^{t-1},x)$, hence $f(B)$ is at least $f(x)=d(C_1^{t-1},x)-d(R^{t-1},x)$.
\end{myproof}

 Let $R^{t_i},\dots,R^t$ be a sequence of $t-t_i$ steps such that $R^t\in \overline{C_1^{t-1}B}$ and $R^k$ is in the interior of the triangle $C_1^{k-1}BB^k$ for $k=t_i,\dots,t-1$ (given the prescribed strategy of cop $c_1$). We define
\begin{align*}
g(t):=\max_{R^{t_i},\dots,R^t, \tau}d(C_1^t,R^t).
\end{align*}
Note that the maximum is well-defined since $S$ is compact and we can assume $\tau(k)\le D$. Hence let $R^{t_i},\dots,R^t$ be a sequence such that $d(C_1^t,R^t)= g(t)$. 

We show that $g(t)\leq g(t-1)$ for every $t \in \{t_i, \ldots, t_{i+1}\}$. If the last step of $c_1$ is of type $(b)$, the robber is caught, and $0=g(t)\leq g(t-1)$. Thus the following holds.
\begin{observation}
    \label{obs:laststep}
    The last step of the cop $c_1$ is of type $(a)$.
\end{observation}

\begin{claim}
    \label{clm:if-a}
    If the second to last step is of type $(a)$, then $g(t) \leq g(t-1)$.
\end{claim}

\begin{myproof}
We can assume $R^t=B$ by Claim~\ref{cl:fmax}. Suppose the second last step from $C_1^{t-2}$ to $C_1^{t-1}$ is of type $(a)$. Note that the cop moves distance exactly $d(R^{t-2},R^{t-1})+d(R^{t-1},B)$ along $\overline{C_1^{t-2}B}$ and for the robber's steps it holds that $d(R^{t-2},R^{t-1})+d(R^{t-1},B)\geq d(R^{t-1},B)$. Hence by maximality of $R^{t_i}, \dots, R^t$, we can assume $R^{t-1}\in \overline{R^{t-2}B}$. But then $R^t_i, \dots,R^{t-2},R^t$ is a sequence with the same value assuming $\tau'(t-1)=\tau(t-1)+\tau(t)$, $\tau'(k)=\tau(k)$ if $k\leq t-2$ and $\tau'(k)=\tau(k-1)$ if $k\geq t$. This shows that $g(t)\leq g(t-1)$.

\end{myproof}

\begin{claim}
    \label{clm:not-b}
    The second to last step is of type $(a)$.
\end{claim}

\begin{myproof}
Now assume the step from $C_1^{t-2}$ to $C_1^{t-1}$ is of type $(b)$. Let $z_c$ be the intersection point of $g^{t-1}$ and $\overline{C_1^{t-2}B}$, see Figure~\ref{fig:complex}. By Lemma~\ref{lem:closertogether} we can assume for the step size at time $t-1$ that $\tau(t-1)=d(R^{t-1},R^{t-2})$,
otherwise changing $\tau(t-1)$ to $\tau'(t-1)=\max \{d(C_1^{t-2},z_c),d(R^{t-1},R^{t-2})\}$ gives a larger value of $g$.

 Let $z_r^1,z_r^2$ be the points on  $g^{t-1}$ such that they are at the same distance from the robber's position $R^{t-2}$ as $z_c$ is from the cop's position $C_1^{t-2}$. Let $z_r^1$ be closer to $B^{t-1}$. Note that $R^{t-1}$ is either in $\overline{z_r^1B^{t-1}}$ or in  $\overline{z_r^2z_c}$. The second case cannot happen since by Lemma~\ref{lem:closertogether}, $R^{t-1}=z_r^2$ would give $C_1^{t-1}=z_c$ and that would be a better choice for the robber, contradicting maximality of $R^{t_i}, \dots, R^t$. We denote $z_r=z_r^1$ and then $R^{t-1}\in \overline{z_rB^{t-1}}$, see Figure~\ref{fig:complex}.  

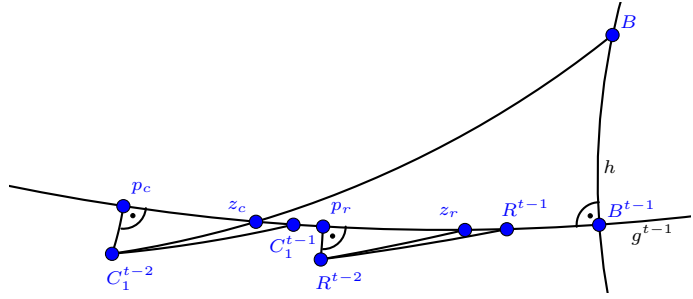
\begin{figure}
    \centering
\begin{tikzpicture}[line cap=round,line join=round,>=triangle 45,x=7cm,y=7cm]
\clip(-0.791397529903244,0) rectangle (0.5019502474256547,0.55);
\draw [shift={(0.03915050606919238,4.295352185700408)},line width=0.8pt]  plot[domain=4.468318698348917:4.938230525743787,variable=\t]({1*4.177509205414962*cos(\t r)+0*4.177509205414962*sin(\t r)},{0*4.177509205414962*cos(\t r)+1*4.177509205414962*sin(\t r)});
\draw [shift={(1.6262130245060902,0.21798746566845137)},line width=0.8pt]  plot[domain=2.619447020646108:3.9302417952789077,variable=\t]({1*1.3008025738988218*cos(\t r)+0*1.3008025738988218*sin(\t r)},{0*1.3008025738988218*cos(\t r)+1*1.3008025738988218*sin(\t r)});
\draw [shift={(-1.0659220087561356,4.98539904225489)},line width=0.8pt]  plot[domain=4.886451004627369:4.942374163786775,variable=\t]({1*4.999039241621004*cos(\t r)+0*4.999039241621004*sin(\t r)},{0*4.999039241621004*cos(\t r)+1*4.999039241621004*sin(\t r)});
\draw [shift={(-1.1111957813343203,0.24293790871332774)},line width=0.8pt]  plot[domain=5.963713277170567:6.1356184643326035,variable=\t]({1*0.5420100478266874*cos(\t r)+0*0.5420100478266874*sin(\t r)},{0*0.5420100478266874*cos(\t r)+1*0.5420100478266874*sin(\t r)});
\draw [shift={(-2.528456918167516,0.2558557064492336)},line width=0.8pt]  plot[domain=6.200072532617923:6.226913956177999,variable=\t]({1*2.3363553945304174*cos(\t r)+0*2.3363553945304174*sin(\t r)},{0*2.3363553945304174*cos(\t r)+1*2.3363553945304174*sin(\t r)});
\draw [shift={(-0.8919574241475181,2.041831463684029)},line width=0.8pt]  plot[domain=4.8612683509464:5.3883939362543005,variable=\t]({1*1.9911463463497459*cos(\t r)+0*1.9911463463497459*sin(\t r)},{0*1.9911463463497459*cos(\t r)+1*1.9911463463497459*sin(\t r)});
\draw [shift={(-0.8246719750337411,2.593921719728593)},line width=0.8pt]  plot[domain=4.8026005216960685:4.940526883262983,variable=\t]({1*2.5315042473766165*cos(\t r)+0*2.5315042473766165*sin(\t r)},{0*2.5315042473766165*cos(\t r)+1*2.5315042473766165*sin(\t r)});
\draw [shift={(-0.9048205002079298,5.5063688743174835)},line width=0.8pt]  plot[domain=4.841099139273715:4.906320048642447,variable=\t]({1*5.4898814484148035*cos(\t r)+0*5.4898814484148035*sin(\t r)},{0*5.4898814484148035*cos(\t r)+1*5.4898814484148035*sin(\t r)});
\begin{scriptsize}
\draw [fill=qqqqff] (-0.5750764303335547,0.16324516483029966) circle (2.5pt) node[above,yshift=2,xshift=7,color=qqqqff]{$p_c$};
\draw [fill=qqqqff] (0.32853523254562444,0.1278781736148007) circle (2.5pt) node[right,yshift=6,color=qqqqff]{$B^{t-1}$};
\draw [fill=qqqqff] (-0.5966107007407843,0.07271133389132173) circle (2.5pt) node[below,yshift=-2,xshift=7,color=qqqqff]{$C_1^{t-2}$};
\draw [fill=qqqqff] (-0.19579954140517178,0.12445520345323935) circle (2.5pt) node[above,yshift=2,xshift=7,color=qqqqff]{$p_r$};
\draw [fill=qqqqff] (-0.20016633954103824,0.06189820844886551) circle (2.5pt) node[below,yshift=-2,xshift=7,color=qqqqff]{$R^{t-2}$};
\draw [fill=qqqqff] (0.07367445922167375,0.11798563989843913) circle (2.5pt)node[above,yshift=2,xshift=-6,color=qqqqff]{$z_r$};
\draw [fill=qqqqff] (0.35386616202452936,0.48858088353948526) circle (2.5pt)node[right, xshift=0.05,yshift=5,color=qqqqff]{$B$};
\draw [fill=qqqqff] (-0.32334713484417127,0.133600316395969) circle (2.5pt)node[above, yshift=1,xshift=-7,color=qqqqff]{$z_c$};
\draw [fill=qqqqff] (-0.2521366764240921,0.12801071437340994) circle (2.5pt)node[below,yshift=-1,color=qqqqff]{$C_1^{t-1}$};
\draw [fill=qqqqff] (0.15317710835218462,0.1193994686191262) circle (2.5pt)node[above,yshift=2,xshift=7,color=qqqqff]{$R^{t-1}$};
\draw[color=black] (0.34863264232841333,0.23868923109440066) node {$h$};
\draw[color=black] (0.43031925181422699,0.110080177970841217) node {$g^{t-1}$};
\draw[line width=0.8pt] (0.32853523254562444,0.1278781736148007)+(91:0.3cm) arc[start angle=91, end angle=185, radius=0.3cm];
\draw[fill=black] (0.31153523254562444,0.1448781736148007)  circle (0.8pt);
\draw[line width=0.8pt] (-0.19579954140517178,0.12445520345323935)+(270:0.3cm) arc[start angle=270, end angle=360, radius=0.3cm];
\draw[fill=black] (-0.17779954140517178,0.10645520345323935)  circle (0.8pt);
\draw[line width=0.8pt] (-0.5750764303335547,0.16324516483029966)+(265:0.3cm) arc[start angle=265, end angle=355, radius=0.3cm];
\draw[fill=black] (-0.5570764303335547,0.14524516483029966)  circle (0.8pt);
\end{scriptsize}
\end{tikzpicture}
    \caption{The case when the second last move of the robber is of type $(b)$.}
    \label{fig:complex}
\end{figure}

In the following we argue that it is of advantage to the robber to move to $z_r$ instead of $R^{t-1}$, which also means that the second to last step is of type $(a)$ as desired.

Let $p_c,p_r$ be the closest points on $g^{t-1}$ from $C_1^{t-2},R^{t-2}$, respectively. This means the angles $\angle C_1^{t-2}p_cB^{t-1}$,  $\angle R^{t-2}p_rB^{t-1}$ are right angles. There are two Lambert quadrilaterals\footnote{A quadrilateral in which three of its angles are right angles.} formed by $C_1^{t-2},p_c,B^{t-1},B^{t-2}$ and $R^{t-2},p_r,B^{t-1},o_h(R^{t-2})\cap h $, respectively. Any two sides of a Lambert quadrilateral determine the length of the other sides, see~\cite{slothers2023hyperbolic}. In particular, since $d(B^{t-1},o_h(R^{t-2})\cap h)\leq d(B^{t-1},B^{t-2})$ and $d(p_r,B^{t-1})<d(p_c,B^{t-1})$, it holds that 
\begin{align*}
d(p_r,R^{t-2})<d(p_c,C_1^{t-2}).
\end{align*}
Therefore by the sine formula for hyperbolic right triangles it holds for the angles that $\angle p_c z_c C_1^{t-2}>\angle p_r z_r R^{t-2}$ and thus \begin{align*}
\angle C_1^{t-2} z_c C_1^{t-1} <\angle R^{t-2} z_r R^{t-1}.\end{align*} 
Since $d(C_1^{t-2},z_c)=d(R^{t-2},z_r)$ and $d(C_1^{t-2},C_1^{t-1})=d(R^{t-2},R^{t-1})$, it holds by Lemma~\ref{lem:intuition} that $d(z_r,R^{t-1})< d(z_c,C_1^{t-1})$.
By Lemmas~\ref{lem:shifttoright} and~\ref{lem:closertogether} we have
\begin{align*}
    d(B,z_c)-d(B,z_r)>d(B,C_1^{t-1})-d(B,R^{t-1}).
\end{align*}
Hence it is of advantage for the robber to move to $z_r$ instead of $R^{t-1}$, a contradiction to the maximality of $R^{t_i}, \dots, R^t$. 
\end{myproof}

We will establish now a positive lower bound on $d(R^{t_i},C_1^{t_i})-d(R^t,C_1^t)$. By Claims~\ref{clm:if-a} and~\ref{clm:not-b}, we can assume that $t=t_i+1$. To compute $d(R^t,C_1^t)$ we only need to compute the difference between $d(C_1^{t_1},B)$ and $d(R^{t_1},B)$. By the Pythagorean Theorem for hyperbolic triangles and $H:=g_0\cap h$, we have
\begin{align*}
& d(R^t,C_1^t)=d(C_1^{t_i},B)-d(R^{t_i},B)\\
= & \acosh( \cosh(d(B,H))\cosh(d(C_1^{t_i}, H)) -
\acosh( \cosh(d(B,H))\cosh(d(R^{t_i}, H))\\
= & \acosh( b\cosh(d(C_1^{t_i}, R^{t_i})+d(R^{t_i}, H))-\acosh( b \cosh(d(R^{t_i}, H))
\end{align*}
where $b=\cosh(d(B,H))=\cosh(8D)>1$. For $x=d(R^{t_i},H)$ and $y=d(C_1^{t_i},H)$, and using that $y-x=d(R^{t_i},C_1^{t_i})\geq\varepsilon$, Lemma~\ref{lem:mean-value} gives $\eta>0$ such that
\begin{align*}
    d(R^t,C_1^t)-d(R^{t_i},C_1^{t_i})\geq \eta\varepsilon.
\end{align*}
By taking $\delta = \eta\varepsilon$, we obtain the conclusion of Lemma \ref{lem:crossing}.
\end{proof}

\section{Catching the Robber}
\label{sec:catch}
We will now turn to the cop catch number. First, we show that on a hyperbolic surface at least three cops are needed to catch the robber. The upper bound is considered in Section~\ref{sec:upper-catch}.

\subsection{Lower bound for cop catch number}
\begin{theorem}
\label{thm:lowerbound}
    If $S$ is a hyperbolic surface, then $c_0(S)\geq 3$. 
\end{theorem}
\begin{proof}
 Let $s=\sys(S)$ be the systolic girth of the surface $S$. The robber chooses the agility function $\tau\equiv \frac{s}{10}$ and starting positions 
$r^0,c_1^0,c_2^0$ where $r^0$ is at distance more than $\frac{s}{10}$ to $c_1^0,c_2^0$. In fact, it suffices to assume that $r^0 \neq c_1^0,c_2^0$.
 We show that if $r^k\neq c_1^k,c_2^k$ then there exists a position $r^{k+1}$ at distance at most $\frac{s}{10}$ from $r^k$ with $d(r^{k+1},c_j^k)> \frac{s}{10}$ ($j=1,2$), and hence the robber can escape from the cops.

If $d(r^k,c_j^k)> \frac{s}{5}$ for $j=1,2$, then the robber's strategy is to stay in the same place, and $d(r^{k+1},c_j^k)> \frac{s}{5}$. If $d(r^k,c_1^k)\leq \frac{s}{5}$ and $d(r^k,c_2^k)> \frac{s}{5}$, then the robber moves away from cop $c_1$;  $r^{k+1}$ is the point at distance $\frac{s}{10}$ on the geodesic through $r^{k},c_1^{k}$ which is further away from $c_1^k$.  We consider $B(r^k,\frac{s}{5})$, the disk of radius at most $\frac{s}{5}$ from $r^k$ on $S$ and this is isometric to a hyperbolic disk by Lemma~\ref{lem:ball-sys}. 
It follows that $d(r^{k+1},c_1^{k})=d(r^{k},c_1^{k})+\frac{s}{10}$. By the triangle inequality $d(r^{k+1},c_2^{k})>\frac{s}{10}$. If $d(r^k,c_1^k) \leq  \frac{s}{5}$ and $d(r^k,c_2^k)> \frac{s}{5}$ we use the same strategy by interchanging the roles of $c_1,c_2$.  Hence we can assume $d(r^k,c_1^k),d(r^k,c_2^k)\leq \frac{s}{5}$. Note that $c_1^k,c_2^k$ are now in the hyperbolic disk $B(r^k,\frac{s}{5})$ of radius $\frac{s}{5}$ centred at $r^k$. We consider the disk $B(r^k,\frac{s}{5})$ being embedded in the Poincaré disk, and consider the geodesic $h$ through $c_1^k,c_2^k$ in this disk and let $o_h(r^k)$ be the orthogonal geodesic to $h$ passing through $r^k$. Now let $r^{k+1}$ be the point at distance $\frac{s}{10}$ on the geodesic $o_h(r^k)$ which is further away from $g$. If $c_j^k$ is on $o_h(r^k)$ then $d(r^{k+1},c_j^k)= d(r^{k},c_j^k)+\frac{s}{10}>\frac{s}{10}$. If $c_j^k$ is not on $o_h(r^k)$, then $\angle c_j^k,o_h(r^k)\cap h, r^{k+1}$ form a right angle and by the hyperbolic Pythagorean theorem $d(c_j^k,r^{k+1})>\frac{s}{10}$.
\end{proof}
\subsection{Upper bound for cop catch number}
\label{sec:upper-catch}
The following lemma shows how to capture the robber on the universal covering space $\mathcal{D}$, if the robber is contained in a polygon which is guarded by the cops. 
\begin{lemma}
\label{lem:capturepolygon}
    Suppose the robber is contained in a bounded convex polygon in the Poincaré disk $\mathcal{D}$ where $n$ cops guard the boundary of the polygon. Then these $n$ cops can catch the robber.
\end{lemma}
\begin{proof}
    Suppose the robber is contained in a bounded convex polygon at the beginning of the game.  For all $k\geq 0$ we consider the bisectors $b_j^k$ between the robber's position $r^k$ and each of the cops' positions $c_j^k$ ($j\in[n]$), and let $B_{j}^k$ be the midpoint between $r^k$ and $c_j^k$, see Figure~\ref{fig:alphaandbeta} and Figure~\ref{fig:closerorfurther}.

    The bisectors $b_j^0$ bound a polygon $P_0$ containing the robber, otherwise there is a point on the boundary of the Poincaré disk towards which the robber can walk and escape from the cops. Without loss of generality let $b_1^0,\dots,b_n^0$ appear consecutively on the convex hull of $P_0$, see Figure~\ref{fig:alphaandbeta} for an example with $n=4$. Let $\alpha_{j,{j+1}}^k$ be the angle between the bisector $b_j^k$ and $b_{j+1}^k$~(we consider the indices modulo $n$) and let $P_{j,j+1}^k$ be their intersection point.
\begin{figure}[ht!]
    \centering
    \begin{tikzpicture}[line cap=round,line join=round,>=triangle 45,x=3cm,y=3cm]
\clip(-1.812639871643102,-1.154331655558786) rectangle (1.812639871643102,1.836877202485103);
\draw [shift={(0.7496347031159181,-0.05185309583904)}, rotate=160, line width=0.8pt,color=qqqqff,fill=qqqqff,fill opacity=0.10000000149011612] (0,0) -- (0:0.20901304132294235) arc (0:31.052884389978164:0.20901304132294235) -- cycle;
\draw [shift={(0.03666216887532403,0.9304238649313455) },rotate=263, shift={(0,0)}, line width=0.8pt,color=qqqqff,fill=qqqqff,fill opacity=0.10000000149011612] (0,0) -- (0:0.20901304132294235) arc (0:8.416425493064619:0.20901304132294235) -- cycle;
\draw [shift={(-0.7320766396776542,0.15448087712926542)}, rotate=-20,line width=0.8pt,color=qqqqff,fill=qqqqff,fill opacity=0.10000000149011612] (0,0) -- (0:0.20901304132294235) arc (0:29.011152645835452:0.20901304132294235) -- cycle;
\draw [shift={(0.07333631683011352,-0.49412760080586177)},rotate=123,line width=0.8pt,color=qqqqff,fill=qqqqff,fill opacity=0.10000000149011612] (0,0) -- (-69.00396261732368:0.20901304132294235) arc (-69.00396261732368:0:0.20901304132294235) -- cycle;
\draw [shift={(-0.0655864222669449,0.11787961616683429)}, rotate=131, line width=0.8pt,color=qqqqff,fill=qqqqff,fill opacity=0.10000000149011612] (0,0) -- (-100.785293508098:0.13934202754862823) arc (-100.785293508098:0:0.13934202754862823) -- cycle;
\draw [shift={(-0.0655864222669449,0.11787961616683429)}, rotate=30,line width=0.8pt,color=qqqqff,fill=qqqqff,fill opacity=0.10000000149011612] (0,0) -- (-81.84982102360783:0.13934202754862823) arc (-81.84982102360783:0:0.13934202754862823) -- cycle;
\draw [shift={(-0.0655864222669449,0.11787961616683429)}, rotate=-53,line width=0.8pt,color=qqqqff,fill=qqqqff,fill opacity=0.10000000149011612] (0,0) -- (-74.08946979673681:0.13934202754862823) arc (-74.08946979673681:0:0.13934202754862823) -- cycle;
\draw [shift={(-0.0655864222669449,0.11787961616683429)},rotate=-128,line width=0.8pt,color=qqqqff,fill=qqqqff,fill opacity=0.10000000149011612] (0,0) -- (-103.27541567155731:0.13934202754862823) arc (-103.27541567155731:0:0.13934202754862823) -- cycle;
\draw [rotate around={0:(0,0)},line width=0.8pt,dash pattern=on 1pt off 1pt,color=aqaqaq] (0,0) ellipse (3cm and 3cm);
\draw [shift={(-0.8465681685600726,1.0366816706293416)},line width=0.8pt]  plot[domain=4.553420895127159:6.240935675203007,variable=\t]({1*0.8895990952322275*cos(\t r)+0*0.8895990952322275*sin(\t r)},{0*0.8895990952322275*cos(\t r)+1*0.8895990952322275*sin(\t r)});
\draw [shift={(1.1099779546217348,0.95958645942069)},line width=0.8pt]  plot[domain=3.104582885430488:4.604317389988864,variable=\t]({1*1.0737118947137472*cos(\t r)+0*1.0737118947137472*sin(\t r)},{0*1.0737118947137472*cos(\t r)+1*1.0737118947137472*sin(\t r)});
\draw [shift={(-1.3752033228946468,-1.4684921385334626)},line width=0.8pt]  plot[domain=0.2979984307316475:1.3383853492126891,variable=\t]({1*1.7457529436421007*cos(\t r)+0*1.7457529436421007*sin(\t r)},{0*1.7457529436421007*cos(\t r)+1*1.7457529436421007*sin(\t r)});
\draw [shift={(0.9660605829778759,-1.121011707762229)},line width=0.8pt]  plot[domain=1.5401100971582198:3.0240642161751192,variable=\t]({1*1.090843847176828*cos(\t r)+0*1.090843847176828*sin(\t r)},{0*1.090843847176828*cos(\t r)+1*1.090843847176828*sin(\t r)});
\draw [shift={(12.889374335020616,11.49025237173576)},line width=0.8pt]  plot[domain=3.830144984692494:3.862028450971523,variable=\t]({1*17.23838363404374*cos(\t r)+0*17.23838363404374*sin(\t r)},{0*17.23838363404374*cos(\t r)+1*17.23838363404374*sin(\t r)});
\draw [shift={(-3.3491277214068362,2.455397258631688)},line width=0.8pt]  plot[domain=5.5216009154927335:5.664524490805949,variable=\t]({1*4.030587077832674*cos(\t r)+0*4.030587077832674*sin(\t r)},{0*4.030587077832674*cos(\t r)+1*4.030587077832674*sin(\t r)});
\draw [shift={(17.062894164322845,13.812335211542337)},line width=0.8pt]  plot[domain=3.816037026226112:3.8558982228823986,variable=\t]({1*21.929955797011875*cos(\t r)+0*21.929955797011875*sin(\t r)},{0*21.929955797011875*cos(\t r)+1*21.929955797011875*sin(\t r)});
\draw [shift={(-1.9127114733500297,3.2545973719304184)},line width=0.8pt]  plot[domain=5.244585895256138:5.42296255980432,variable=\t]({1*3.6401743411080223*cos(\t r)+0*3.6401743411080223*sin(\t r)},{0*3.6401743411080223*cos(\t r)+1*3.6401743411080223*sin(\t r)});
\begin{scriptsize}
\draw [fill=qqqqff] (-0.0655864222669449,0.11787961616683429) circle (2.5pt);
\draw[color=qqqqff] (-0.0550887891737795,0.18425311304034063) node {$r$};
\draw [fill=qqqqff] (-0.42153149639567145,0.5366385269065145) circle (2.5pt);
\draw[color=qqqqff] (-0.36718728816467683,0.6349831080831086) node {$c_1$};
\draw[color=black] (-0.16281898109335538,0.6153120943087944) node {$b_1$};
\draw [fill=qqqqff] (0.03666216887532403,0.9304238649313455) circle (2.5pt);
\draw[color=qqqqff] (0.18266768226848957,0.9211540234190187) node {$P_{1,2}$};
\draw[color=black] (0.16695715077173143,0.5967331573023106) node {$b_2$};
\draw [fill=qqqqff] (0.46166011000614604,0.49538936639478803) circle (2.5pt);
\draw[color=qqqqff] (0.5478253594046487,0.5838910313152781) node {$c_2$};
\draw [fill=qqqqff] (-0.432016200152546,-0.3259859520311673) circle (2.5pt);
\draw[color=qqqqff] (-0.33578978196614349,-0.39600820985316043) node {$c_4$};
\draw[color=black] (-0.44139791809983916,-0.04396374224108064) node {$b_4$};
\draw [fill=qqqqff] (0.49375246615211626,-0.553823100436365) circle (2.5pt);
\draw[color=qqqqff] (0.5803384991659952,-0.4751083780542975) node {$c_3$};
\draw[color=black] (0.486351737365746,-0.20358219697000833) node {$b_3$};
\draw [fill=qqqqff] (-0.7320766396776542,0.15448087712926542) circle (2.5pt);
\draw[color=qqqqff] (-0.743410762545973,0.28591750771477947) node {$P_{4,1}$};
\draw [fill=qqqqff] (0.07333631683011352,-0.49412760080586177) circle (2.5pt);
\draw[color=qqqqff] (0.2287597590363199,-0.506910986318886) node {$P_{3,4}$};
\draw [fill=qqqqff] (0.7496347031159181,-0.05185309583904) circle (2.5pt);
\draw[color=qqqqff] (0.7382594637211072,0.07581238962400653) node {$P_{2,3}$};
\draw[color=qqqqff] (0.43073328263681826,-0.03078611010131328) node {$\alpha_{2,3}$};
\draw[color=qqqqff] (0.0203519342425545,0.5367331573023106) node {$\alpha_{1,2}$};
\draw[color=qqqqff] (-0.40718728816467683,0.14941760615318355) node {$\alpha_{4,1}$};
\draw[color=qqqqff] (0.05515454250714297,-0.233585257119036) node {$\alpha_{3,4}$};
\draw[color=qqqqff] (-0.03454275078986433,0.3184306474761261) node {$\beta_{1,2}$};
\draw[color=qqqqff] (0.1687597590363199,0.0547728719015626) node {$\beta_{2,3}$};
\draw[color=qqqqff] (-0.0348980165382433855,-0.08424016942137996) node {$\beta_{3,4}$};
\draw[color=qqqqff] (-0.28642419762253237,0.0547728719015626) node {$\beta_{4,1}$};
\end{scriptsize}
\end{tikzpicture}
  \caption{A schematic figure when the number of cops is $n=4$. A label $x$ in the figure denotes the value $x^k$. That is, the figure depicts the positions of the cops $c_i^k$ ($i=1,2,3,4$), the intersection points $P_{i,i+1}^k$ of the bisectors $b_i^k, b_{i+1}^k$, and the angles $\beta_{i,i+1}^k,\alpha_{i,i+1}^k$.}
    \label{fig:alphaandbeta}
\end{figure}
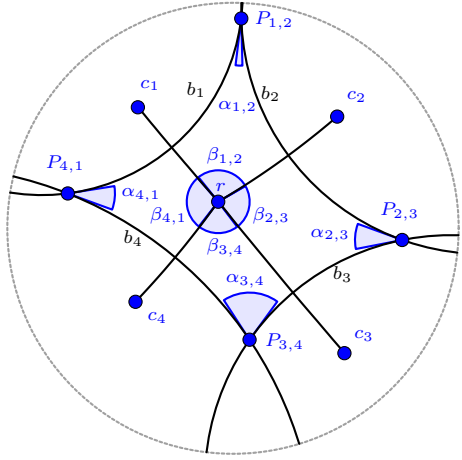

    Every cop has the same strategy, which is as follows. Cop $c_j$ copies the move of the robber by reflecting it along the bisector $b_j$ and subsequently if there is any agility left she moves as close to the robber as possible (see Figure~\ref{fig:closerorfurther}). 

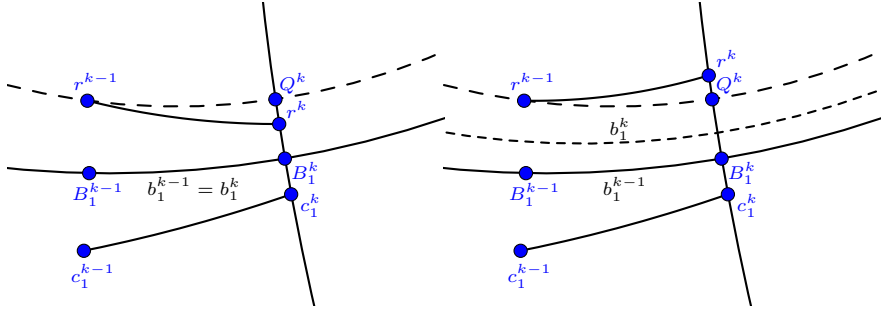
\begin{figure}[ht!]
    \centering
   \begin{tikzpicture}[line cap=round,line join=round,>=triangle 45,x=7.5cm,y=7.5cm]
\clip(-0.4414271878455544,0.0027800505625030455) rectangle (0.33274657450456,0.5395553123591877);
\draw [shift={(-0.25934590292717835,2.08769115197294)},line width=0.8pt]  plot[domain=4.3406266544369565:5.3313373261774535,variable=\t]({1*1.8508686186197052*cos(\t r)+0*1.8508686186197052*sin(\t r)},{0*1.8508686186197052*cos(\t r)+1*1.8508686186197052*sin(\t r)});
\draw [shift={(4.92055325520448,1.0902596031607736)},line width=0.8pt]  plot[domain=3.159899931587523:3.559384925808607,variable=\t]({1*4.939687271436086*cos(\t r)+0*4.939687271436086*sin(\t r)},{0*4.939687271436086*cos(\t r)+1*4.939687271436086*sin(\t r)});
\draw [shift={(-0.14171566499531968,1.556803051113228)},line width=0.8pt,dash pattern=on 5pt off 5pt]  plot[domain=4.1090646281454:5.497272791724265,variable=\t]({1*1.201548613107485*cos(\t r)+0*1.201548613107485*sin(\t r)},{0*1.201548613107485*cos(\t r)+1*1.201548613107485*sin(\t r)});
\draw [shift={(0.03564646681252762,1.702028670773432)},line width=0.8pt]  plot[domain=4.4667455209913935:4.7159943413877174,variable=\t]({1*1.377741727150267*cos(\t r)+0*1.377741727150267*sin(\t r)},{0*1.377741727150267*cos(\t r)+1*1.377741727150267*sin(\t r)});
\draw [shift={(-0.8672319840916518,2.882419688698807)},line width=0.8pt]  plot[domain=4.911529194157058:5.045765017520871,variable=\t]({1*2.8390904487229824*cos(\t r)+0*2.8390904487229824*sin(\t r)},{0*2.8390904487229824*cos(\t r)+1*2.8390904487229824*sin(\t r)});
\begin{scriptsize}
\draw [fill=qqqqff] (-0.29939348076326583,0.36564530502114523) circle (2.5pt);
\draw[color=qqqqff] (-0.283877105930003,0.3997731891103237) node {$r^{k-1}$};
\draw [fill=qqqqff] (-0.3055843459275345,0.0994381029575836) circle (2.5pt);
\draw[color=qqqqff] (-0.28971161964518434,0.05488606408991622) node {$c_1^{k-1}$};
\draw[color=black] (-0.11049858083448855,0.2090829148066966) node {$b_1^{k-1}=b_1^k$};
\draw [fill=qqqqff] (0.033930152313511666,0.3681619713408548) circle (2.5pt);
\draw[color=qqqqff] (0.06552367779464543,0.3959396871510121) node {$Q^k$};
\draw [fill=qqqqff] (0.04061371234659372,0.32429589798953506) circle (2.5pt);
\draw[color=qqqqff] (0.0721916934691384,0.35093159310405424) node {$r^k$};
\draw [fill=qqqqff] (0.06181792014930021,0.19964111947995633) circle (2.5pt);
\draw[color=qqqqff] (0.09636325028917541,0.17990629920731089) node {$c_1^k$};
\draw [fill=qqqqff] (0.050636024408627814,0.26296491274788997) circle (2.5pt);
\draw[color=qqqqff] (0.08719371698087785,0.2374184402777477) node {$B_1^k$};
\draw [fill=qqqqff] (-0.2961447908716764,0.23718838644210624) circle (2.5pt);
\draw[color=qqqqff] (-0.2805430980927565,0.2000798795390874) node {$B_1^{k-1}$};
\end{scriptsize}
\end{tikzpicture}%
\begin{tikzpicture}[line cap=round,line join=round,>=triangle 45,x=7.5cm,y=7.5cm]
\clip(-0.4414271878455544,0.0027800505625030455) rectangle (0.33274657450456,0.5395553123591877);
\draw [shift={(-0.25934590292717835,2.08769115197294)},line width=0.8pt]  plot[domain=4.3406266544369565:5.3313373261774535,variable=\t]({1*1.8508686186197052*cos(\t r)+0*1.8508686186197052*sin(\t r)},{0*1.8508686186197052*cos(\t r)+1*1.8508686186197052*sin(\t r)});
\draw [shift={(4.92055325520448,1.0902596031607736)},line width=0.8pt]  plot[domain=3.159899931587523:3.559384925808607,variable=\t]({1*4.939687271436086*cos(\t r)+0*4.939687271436086*sin(\t r)},{0*4.939687271436086*cos(\t r)+1*4.939687271436086*sin(\t r)});
\draw [shift={(-0.14171566499531968,1.556803051113228)},line width=0.8pt,dash pattern=on 5pt off 5pt]  plot[domain=4.1090646281454:5.497272791724265,variable=\t]({1*1.201548613107485*cos(\t r)+0*1.201548613107485*sin(\t r)},{0*1.201548613107485*cos(\t r)+1*1.201548613107485*sin(\t r)});
\draw [shift={(-0.8672319840916518,2.882419688698807)},line width=0.8pt]  plot[domain=4.911529194157058:5.045765017520871,variable=\t]({1*2.8390904487229824*cos(\t r)+0*2.8390904487229824*sin(\t r)},{0*2.8390904487229824*cos(\t r)+1*2.8390904487229824*sin(\t r)});
\draw [shift={(-0.28045520677697783,1.4432019361444353)},line width=0.8pt]  plot[domain=4.694815589003208:5.002520685854346,variable=\t]({1*1.0777230402563331*cos(\t r)+0*1.0777230402563331*sin(\t r)},{0*1.0777230402563331*cos(\t r)+1*1.0777230402563331*sin(\t r)});
\draw [shift={(-0.19670045423428295,1.804960079854035)},line width=0.8pt,dash pattern=on 3pt off 3pt]  plot[domain=4.237654221123286:5.404222474764782,variable=\t]({1*1.5154444755789167*cos(\t r)+0*1.5154444755789167*sin(\t r)},{0*1.5154444755789167*cos(\t r)+1*1.5154444755789167*sin(\t r)});
\begin{scriptsize}
 \draw [fill=qqqqff] (-0.29939348076326583,0.36564530502114523) circle (2.5pt);
\draw[color=qqqqff] (-0.283877105930003,0.3997731891103237) node {$r^{k-1}$};
\draw [fill=qqqqff] (-0.3055843459275345,0.0994381029575836) circle (2.5pt);
\draw[color=qqqqff] (-0.28971161964518434,0.05488606408991622) node {$c_1^{k-1}$};
\draw[color=black] (-0.12049858083448855,0.2090829148066966) node {$b_1^{k-1}$};
\draw [fill=qqqqff] (0.033930152313511666,0.3681619713408548) circle (2.5pt);
\draw[color=qqqqff] (0.06552367779464543,0.3959396871510121) node {$Q^k$};
\draw [fill=qqqqff] (0.06181792014930021,0.19964111947995633) circle (2.5pt);
\draw[color=qqqqff] (0.09636325028917541,0.17990629920731089) node {$c_1^k$};
\draw [fill=qqqqff] (0.050636024408627814,0.26296491274788997) circle (2.5pt);
\draw[color=qqqqff] (0.08719371698087785,0.2374184402777477) node {$B_1^k$};
\draw [fill=qqqqff] (-0.2961447908716764,0.23718838644210624) circle (2.5pt);
\draw[color=qqqqff] (-0.2805430980927565,0.2000798795390874) node {$B_1^{k-1}$};
\draw [fill=qqqqff] (0.02785810702835706,0.41052103091294656) circle (2.5pt);
\draw[color=qqqqff] (0.060273424855678,0.4408660498114304) node {$r^k$};
\draw[color=black] (-0.12949858083448855,0.31542704020264045) node {$b_1^{k}$};
\end{scriptsize}
\end{tikzpicture}
    \caption{A step $K_1^+$ in comparison to a step $K_1^{-}$. In both cases the distance of $r^k$ to $r^{k-1}$ is the same and hence the move of cop $c_1$ is the same.
    }
    \label{fig:closerorfurther}
\end{figure}

    More precisely, in the $k$-th round the cop $c_j^{k-1}$ will move to $o_{b_j^{k-1}}(r^k)$, as close to the robber as possible. In particular, $b_j^{k-1}$ and $b_j^k$ are parallel and the voronoi cell of the robber with respect to $b_j^k$ is contained in the voronoi cell with respect to $b_j^{k-1}$. Therefore throughout the whole game, the robber is restricted to $P_0$.
    We have to show that with this strategy the robber is eventually caught. If
\begin{align*}
    \sum_{k =1}^\infty \tau(k)-d(r^k,r^{k-1})=\infty,
\end{align*}
then the cops catch the robber. Note that while the robber is not caught, the distance between $b_j^{k-1}$ and $b_j^k$ is at least $\tau(k)-d(r^k,r^{k-1})$. The part of the Poincaré disk $\mathcal{D}\setminus b_j^k$ containing the cop $c_j$ eventually includes $P_0$, in which $r^k$ is contained, which is a contradiction since $b_j^k$ is a bisector.

Hence we assume that the robber moves an infinite distance, that is 
\begin{align*}
    \sum_{k =1}^\infty d(r^k,r^{k-1})=\infty.
\end{align*}
We want to show that also in this case one of the bisectors surpasses the polygon. Note that for some $j$, $\alpha_{j,j+1}^0<(n-2)\pi/n$ since the angle sum in a hyperbolic $n$-gon is upper bounded by the angle sum in an Euclidean $n$-gon. We suppose $j=1$ in the following and will show, that either $b_1^k$ or $b_2^k$ eventually surpasses the polygon. By the following claim we can assume that also $\alpha_{1,2}^k<(n-2)\pi/n$ for all positive integers $k$, as all the angles are non-decreasing and one of them has to stay below $(n-2)\pi/n$.  
    \begin{claim} The angles in the guarded polygon  are increasing, which means
    \label{cl:increasingAngle}
    \begin{align*}
        \alpha_{j,j+1}^k \geq \alpha_{j,j+1}^{k-1}.
    \end{align*}
        \end{claim}
        \begin{myproof}
            We consider the step of the cop $c_j$ and the cop $c_{j+1}$ separately, imagining the cop $c_j$ takes her turn first. Let $P'$ be the intersection point between $b_j^k,b_{j+1}^{k-1}$ and let $\alpha'$ be the angle between $b_j^k,b_{j+1}^{k-1}$.
         If $b_j^k=b_j^{k-1}$, then $\alpha'=\alpha_{j,j+1}^{k-1}$. Otherwise $b_j^k\neq b_j^{k-1}$ and note that $b_j^k$ is orthogonal to $o_{b_j^{k-1}}(r^k)$. In particular, $b_j^{k-1}$ and $b_j^k$ are parallel, see Figure~\ref{fig:closerorfurther}. Since the quadrilateral $B_j^k ,B_j^{k-1},P_{j,j+1}^{k-1}, P'$ has angle sum smaller than $2\pi$ and has a right angle at $B_j^{k-1}$ and $B_j^{k}$, it holds that the angle at $P'$ in the quadrilateral is smaller than $\pi-\alpha_{j,j+1}^{k-1}$ but then $\alpha'> \alpha_{j,j+1}^{k-1}$.
        Now we consider the quadrilateral formed by $B_{j+1}^k,B_{j+1}^{k-1},P', P_{j,j+1}^k$. By the same argument, $\alpha_{j,j+1}^k>\alpha'$. 
        \end{myproof}

 We want in the following that the angle between $\overline{r^{k-1}r^{k}}$ and $\overline{r^kc_j^k}$ for $j=1$ or $j=2$ is bounded away from $\pi/2$, as this will help the cop $c_1$ or $c_2$ to get closer to the robber. We show that if the angle between $\overline{r^{k-1}r^{k}}$ and $\overline{r^kc_1^k}$ is close to $\pi/2$, then the angle between $\overline{r^{k-1}r^{k}}$ and $\overline{r^kc_2^k}$ is bounded away from $\pi/2$ by bounding the angle $\beta_{1,2}^k$, see Figure \ref{fig:alphaandbeta}. We define the angle $\beta_{1,2}^k$ for every $k\in \mathbb{N}$ as the angle between $\overline{r^kc_{1}^k}$ and $\overline{r^kc_{2}^k}$.

\begin{claim}
\label{cl:boundingbeta}
It holds that
    \begin{align*}
        \acot \left(\frac{\cosh(d\left(P_{1,2}^{k},r^k\right))}{\cot((n-2)\pi/n)}\right)\leq \beta_{1,2}^k \leq \pi-\alpha_{1,2}^k.
    \end{align*}
\end{claim}
\begin{myproof}
    Since $\alpha_{1,2}^k\leq (n-2) \pi/2$, there is a quadrilateral formed by $r^kB_1^kP_{1,2}^{k}B_{2}^k$ with right angles at $B_1^k$ and $B_{2}^k$. Since a quadrilateral has angle sum at most $2\pi$ it holds that $\beta_{1,2}^k \leq  \pi-\alpha_{1,2}^k$. On the other hand~(\cite[Corollary 32.13]{MR0666074}),
 $$ \cot\left(\angle P_{1,2}^{k}r^kB_1^k\right) =\frac{\cosh\left(d\left(P_{1,2}^{k},r^k\right)\right)}{\cot\left(\angle r^kP_{1,2}^{k}B_1^k\right)}\leq \frac{\cosh\left(d\left(P_{1,2}^{k},r^k\right)\right)}{\cot\left(\left(n-2\right)\pi/n\right)}.$$ 
Since
$\beta_{1,2}^k \geq \angle P_{1,2}^{k}r^kB_1^k$ the claim follows.
\end{myproof}
Note that $\alpha_{1,2}^k \geq \alpha_{1,2}^0$ by Claim~\ref{cl:increasingAngle} and $d\left(P_{1,2}^{k},r^k\right)$ is upper bounded by the diameter $D$ of $P_0$. Let
$$\beta=\min\left\{\acot \left(\frac{\cosh(D)}{\cot((n-2)\pi/n)}\right), \alpha_{1,2}^0\right\}.$$
 If the angle at $r^k$ between $\overline{r^{k-1}r^{k}}$ and $\overline{r^kc_1^k}$ is in $[\frac{\pi}{2}-\frac{\beta}{2}, \frac{\pi}{2}+\frac{\beta}{2}]$, then the angle between $\overline{r^{k-1}r^{k}}$ and $\overline{r^kc_2^k}$ is in $[0,\frac{\pi}{2}-\frac{\beta}{2}]\cup [\frac{\pi}{2}+\frac{\beta}{2}, \pi]$ by Claim~\ref{cl:boundingbeta} and the definition of $\beta$. Let $K_j$ for $j=1,2$ be the set of rounds at which the angle between $\overline{r^{k-1}r^{k}}$ and $\overline{r^kc_j^k}$ is not in  $[\frac{\pi}{2}-\frac{\beta}{2}, \frac{\pi}{2}+\frac{\beta}{2}]$. Without loss of generality, 
\begin{align*}
    \sum_{k \in K_1} d(r^k,r^{k-1})=\infty.
\end{align*}
We will show that if the cop $c_1$ does not catch the robber, then $b_1^k$ eventually surpasses $P_0$, which is a contradiction.  Let $Q^k$ be the closest point to $r^{k-1}$ on the geodesic $r^kc_1^k$. Let $K_1^+$ be the time steps in $K_1$ in which the robber's position $r^k$ is closer to $c_1^k$ than $Q^k$ to $c_1^k$ and let $K_1^{-}$ be the remaining steps, see Figure~\ref{fig:closerorfurther}. If $k$ is a step of type $K_1^{-}$, then the distance between $b_1^{k-1}$ and $b_1^{k}$ is at least $d(r^k,Q^k)$.
 Suppose $k$ is a step of type $K_1^{+}$.
Since $r^{k-1}Q^kB_1^kB_1^{k-1}$ is a Lambert quadrilateral with acute angle at $r^{k-1}$, it holds that $d(r^{k-1},B_1^{k-1})\leq d(Q^k,B_1^k)$. Hence $d(r^k,Q^k)$
is the distance that the cop $c_1$ is getting closer to $r$ in step $k$. We will show that 
\begin{align*}
    \sum_{k \in K_1} d(r^k,Q^k)=\infty,
\end{align*}
which then proves that the robber is eventually caught.
Note that the triangle $r^{k-1}Q^kr^k$ has a right angle at $Q^k$ and an angle at $r^k$ not in  $[\frac{\pi}{2}-\frac{\beta}{2}, \frac{\pi}{2}+\frac{\beta}{2}]$. In fact, this angle is at most $\frac{\pi}{2}-\frac{\beta}{2}$, since it is in a right-angled triangle. By the cosine formula for right hyperbolic triangles~(see~\cite{slothers2023hyperbolic}), it follows that
\begin{align*}
d\left(r^k,Q^k\right)\geq \atanh\left(\cos\left(\frac{\pi}{2}-\frac{\beta}{2}\right)\tanh\left( d\left(r^k,r^{k-1}\right)\right)\right).
\end{align*}
Using that $\tanh(x)$ is monotonly increasing and for small $x$, $\tanh(x)\geq x\left(1-x^2/3\right)$, we get that
\begin{align*}
    \sum_{k \in K_1} \tanh\left(d\left(r^k,r^{k-1}\right)\right)=\infty.
\end{align*}
But then, since $\atanh(x)\geq x$, we also have 
\begin{align*}
    \sum_{k \in K_1} d(r^k,Q^k) \geq \sum_{k \in K_1} \atanh\left(\cos\left(\frac{\pi}{2}-\frac{\beta}{2}\right)\tanh\left( d\left(r^k,r^{k-1}\right)\right)\right)=\infty,
\end{align*}
thus the bisector $b_1^k$ eventually surpasses $P_0$, a contradiction, or the distance between the cop $c_1$ and robber decreases to $0$.  
\end{proof}

We now use Lemma~\ref{lem:capturepolygon} to bound $c_0$ from above for the special surfaces $S(g),S'(g)$ and $N(g)$. 
\begin{theorem} \label{thm:catch}
If $g\geq 2$, then  
$c_0(S(g))\leq 5$,  $c_0(S'(g))\leq 6$ and $c_0(N(g))\leq 4$. 
\end{theorem}

\begin{proof}
 We will give the proof only for $S(g)$, the proof for the other surfaces is similar.  Let $O$ be the midpoint of the fundamental polygon $P\left(4g, \frac{2\pi}{4g}\right)$. We will play the game in the covering space and choose the player's positions such that they are in $P\left(4g, \frac{2\pi}{4g}\right)$. 
We will first use the cops $c_1,c_2,c_3$ to guard isometric paths. We start by moving cop $c_1$ to the isometric path $\overline{Ov_1}$, cop $c_2$ to the isometric path $\overline{Ov_5}$ and cop $c_3$ to the isometric path $\overline{Ov_9}$. By Lemma~\ref{lem:guardingpath} we can assume that after a finite amount of time the cops guard the respective isometric paths.  Now if the robber is in one of the triangles $ Ov_jv_{j+1}$ for some $1\leq j \leq 4$, then the robber's moves are restricted to the specified triangles since $a_{1}=a_{3}^{-1}$ and $a_{2}=a_{4}^{-1}$. Similarly if the robber is contained in one of the triangles  $ Ov_jv_{j+1}$ for some $5\leq j \leq 8$ his moves are restricted to these triangles. If the robber is outside $ Ov_jv_{j+1}$ for some $1\leq j \leq 8$, we move cop $c_2$ to the isometric path $\overline{Ov_{13}}$ and wait until she is guarding it. If the robber is in one of the triangles $Ov_jv_{j+1}$ for $9\leq j \leq 12$, his moves are restricted to these triangles. If the robber is not contained in one of these triangles we keep going for $i=3,4,\dots$ in the same way, moving the cop currently guarding $\overline{Ov_{1+4i}}$ to guard $\overline{Ov_{1+4(i+2)}}$ unless $i+2=g$, in which case the robber is contained in one of $Ov_{4g-3}v_{4g-2},Ov_{4g-2}v_{4g-1}, Ov_{4g-1}v_{4g}$ or $Ov_{4g}v_{1}$, cop $c_1$ guards $Ov_1$ and one of $c_2$ or $c_3$ guards $Ov_{4g-3}$.
 
 Hence we can assume without loss of generality that cop $c_1$ guards $\overline{Ov_1}$, cop $c_2$ guards $\overline{Ov_5}$ and the robber is in one of the triangles $ Ov_jv_{j+1}$ for some $1\leq j \leq 8$. Cop $c_3,c_4,c_5$ will guard $\overline{Ov_2},\overline{Ov_3},\overline{Ov_4}$, respectively. Now the robber is captured in either $R_1=Ov_1v_2\cup Ov_3v_4$ or $R_2=Ov_2v_3\cup Ov_4v_5$. The regions $R_1$ and $R_2$ can be embedded in the covering space $\mathcal{D}$ such that they form a quadrilateral which is guarded by four of the cops. By Lemma~\ref{lem:capturepolygon} we can now catch the robber.
\end{proof}

\section{Generalizations}
\label{sec:generalizations}
The proof of Theorem~\ref{thm:mainthm} can be generalized to higher dimensional hyperbolic manifolds. We first prove the following lemma, which will replace the isometric path lemma in the proof of Theorem~\ref{thm:mainthm}.  
\begin{lemma}
\label{lem:ball-guard}
Let the Cops and Robber game be played on an Euclidean disk, a disk on the unit sphere of radius at most $\frac{\pi}{2}$ or a hyperbolic disk $B^2$. If the cops position $c^{k_0}$ is on the line between the center $O$ of $B^2$ and the position of the robber $r^{k_0}$, then the cop can guard the disk of radius $\frac{1}{2}(d(O, c^{k_0})+d(O, r^{k_0}))$ and center $O$.
\end{lemma}

\begin{proof}
Let $d_k=\frac{1}{2}(d(O, c^{k})+d(O, r^{k}))$ and $D_k$ be the disk of radius $d_k$ and center $O$. We explain the cop's strategy for the time steps $k\geq k_0$. First suppose that the geodesic segment from $r^{k}$ to $r^{k+1}$ passes through $D_k$, we show that the cop can catch the robber. We are allowed to subdivide the robber's step since this is only to the advantage of the robber, so we assume without loss of generality $d(O,r^{k+1})= d_k$. We can assume $r^{k+1}\neq c^{k}$ and we consider the closest point $p$ from $r^{k+1}$ on the geodesic through $O$ and $r^{k}$.
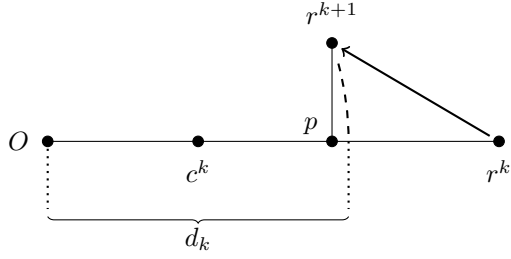
\begin{figure}[htb]
    \centering
    \begin{tikzpicture}
    \draw (0,0) -- (6,0);
    \filldraw[black] (0,0) circle (2pt) node[label=left:$O$] (o) {};
    \filldraw[black] (6,0) circle (2pt) node[label=below:$r^k$] (rn) {};
    \filldraw[black] (3.779645, 1.3093) circle (2pt) node[label=above:$r^{k+1}$] (rn1) {};
        \filldraw[black] (3.779645, 0) circle (2pt);
    \draw (3.5, 0.2) node {$p$};
    \filldraw[black] (2,0) circle (2pt) node[label=below:$c^k$] (cn) {};
    \draw[thick, ->] (rn) -- (rn1);
    \draw (3.779645, 1.3093) -- (3.779645, 0);

     \draw [decorate,
    decoration = {brace}] (4,-1) -- (0,-1) node [midway,yshift=-0.3cm] {$d_k$};
    \draw[thick, dotted] (4,0) -- (4,-0.9);
    \draw[thick, dotted] (0,0) -- (0,-0.9);
    \draw[thick, dashed] (15:4) arc (15:0:4);
    \end{tikzpicture}
    \caption{Depicted is the point $p$ which is the closest point on $Or^k$ from the robber's position $r^{k+1}$.}
    \label{fig:catchingrobber}
\end{figure}

By the Pythagorean theorem (in the Euclidean or hyperbolic plane or for a right angled triangle on a spherical hemisphere), 
\begin{align*}
    d(O,p)<d_k,
\end{align*}
hence $d(r^{k},p)>d(c^{k},p)$.
But then using the Pythagorean theorem again on the triangles $c^{k}pr^{k+1}$ and $r^{k+1}pr^{k}$ which both have a right angle at $p$, it follows that $d(r^{k+1},c^{k})<d(r^{k+1},r^{k})$, so the cop can catch the robber.

We show now that if $d(O,r^{k+1})>d_k$ and the cop does not pass through $D$, there exists a position of the robber on $\overline{Or^{k+1}}$ such that 
\begin{align}
\label{eq:copcanguard}
    d(O,c^{k+1})-d(O,c^{k})\geq d(O,r^{k})-d(O,r^{k+1}).
\end{align}
By subdividing the robber's step, we can assume that the triangle $Or^{k}r^{k+1}$ has angle at most $\frac{\pi}{2}$ at $O$.
Suppose first the angle at $r^{k}$ in the triangle $Or^{k}r^{k+1}$ is smaller than $\pi$. The cop moves to the point $c^{k+1}$ on $\overline{Or^{k+1}}$ such that
\begin{align*}
       d(O,c^{k+1})-d(O,c^{k})=d(O,r^{k})-d(O,r^{k+1}).
\end{align*}
We show now that this is a valid move, which means $d(c^{k},c^{k+1})\leq d(r^{k},r^{k+1})$.

\begin{figure}[htb]
    \centering
    \begin{tikzpicture}
    \draw (0,0) -- (6,0);
    \draw (0,0) -- (6,2.078460969);
    \filldraw[black] (0,0) circle (2pt) node[label=left:$O$] (o) {};
    \filldraw[black] (6,0) circle (2pt) node[label=below:$r^k$] (rn) {};
    \filldraw[black] (5, 1.73205) circle (2pt) node[label=above:$r^{k+1}$] (rn1) {};
    \filldraw[black] (5, 0) circle (2pt) node[label=below:$p_r$] (pr) {};
    \filldraw[black] (2,0) circle (2pt) node[label=below:$c^k$] (cn) {};
    \filldraw[black] (0.622036*5, 0.622036*1.73205) circle (2pt) node[label=above:$c^{k+1}$] (cn1) {};
    \filldraw[black] (0.622036*5, 0) circle (2pt) node[label=below:$p_c$] (pc) {};
    \draw[thick, ->] (rn) -- (rn1);
    \draw[thick, ->] (cn) -- (cn1);

    \end{tikzpicture}
    \caption{The case where the angle at $r^{k}$ in the triangle $Or^kr^{k+1}$ is smaller than $\pi$.}
    \label{fig:onemovesmallerpi}
\end{figure}
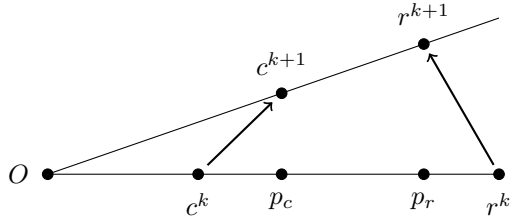

If $r^{k+1}$ is on the geodesic $Or^{k}$, then the cop moves as far as the robber, hence this is a valid move for the cop. Otherwise consider the closest points $p_c,p_r$  on the geodesic $Or^{k}$ of $c^{k+1},r^{k+1}$, respectively. Then using the Pythagorean theorem again, 
$d(O,c^{k+1})>d(O,p_c)$ and $d(O,r^{k+1})>d(O,p_r)$. Therefore it holds that 
$d(p_c,c^{k})< d(O,c^{k+1})-d(O,c^{k})$ and $d(p_r,r^{k})> d(O,r^{k})-d(O,r^{k+1})$. Using the Pythagorean theorem again, $d(c^{k},c^{k+1})<d(r^{k},r^{k+1})$, hence~\ref{eq:copcanguard} is a valid move.

Suppose now the angle at $r^{k}$ in the triangle $Or^{k}r^{k+1}$ is larger than $\pi$. Then this is of advantage to the cop since the robber moves further away from $O$, i.e. $$d(O,r^{k})-d(O,r^{k+1})<0.$$ The cop can simply move at a right angle to $Oc^{k}$ and 
\begin{align*}
    d(O,c^{k+1})-d(O,c^{k})\geq 0,
\end{align*}
which establishes~\eqref{eq:copcanguard}. As long as the cop has not caught the robber, the value $d_k$ is increasing, and as soon as the robber passes through the disk $D_k$, he is caught. This proves the lemma.
\end{proof}

\begin{proof}[Proof of Theorem~\ref{thm:hyperbolicmanifold}.]
Suppose $M=\mathbb{H}^n/\Gamma$ where $\Gamma$ is a torsion-free, discrete group of isometries on $\mathbb{H}^n$. Let $D$ be the diameter of $M$. We position the cop $c_2$ in the covering space such that she is of distance at most $D$ to a point $O$ on the isometric path $\overline{C_1^kR^k}$ at distance between $9D$ and $11D$ from the robber's position (and further from cops $c_1$). 
By Lemma~\ref{lem:ball-guard}, the cop $c_2$ can guard  the $n$-dimensional ball $B$ of radius $4D$ with center $O$ after moving to $O$. We consider the $(n-1)$-dimensional hyperplane $H$ through $O$ which is perpendicular to the geodesic $C_1^kR^k$. Note that cop $c_2$ guards all points on $H$ of distance at most $4D$ to $O$. We denote by $o_H(C_1^k)$ the orthogonal geodesic to $H$ which contains $C_1^k$. Now at the cops' turn, we consider the positions $C_1^k,R^{k+1}$ and $o_H(C_1^k)\cap H$. These points lie on a common $2$-dimensional subspace of $\mathcal{D}^n$ which is isometric to $\mathcal{D}^2$ and is also isometric as a subset of $\mathcal{D}^n$. Now the cop can use the strategies (a) and (b) outlined in Theorem~\ref{thm:mainthm} on this subspace. The proof idea for the lower bound is the same as in the case $n=2$. 
\end{proof}

We are ready to summarise our results about manifolds of constant curvature, i.e.\ \emph{space forms}.
\begin{theorem}
    If $M$ is a compact manifold of constant curvature, then $c(M)= 2$.
\end{theorem}
\begin{proof}
     If $M_{g^p}$ is a Riemannian manifold equipped with a Riemannian metric $g^p$ and constant curvature $C\neq 0$, we can equip the same space with the Riemannian metric ${g^p}^*=|C|\cdot g^p$ instead. This gives a Riemannian manifold $M_{{g^p}^*}$ of constant curvature $1$ or $-1$. Given the step function $\tau$ for $M_{g^p}$, we can play the game with step function $\sqrt{|C|}\cdot \tau$ on $M_{{g^p}^*}$ instead. By rescaling the lengths of the steps appropriately, $k$ cops can win the game on $M_g$ if and only if they can win the game on $M_{{g^p}^*}$.
    The result follows for negative curvature by Theorem~\ref{thm:hyperbolicmanifold}, for constant curvature $0$  by generalizing~\cite[Lemma 8]{MR4517712} and positive constant curvature by generalizing the results for the sphere in~\cite[Theorem 4]{MR4517712}.\footnote{The upper bound follows from the Killing-Hopf theorem and the covering-space method and the lower bound by considering a small enough agility function that allows the robber to move in the opposite direction to the cop without being caught.}
\end{proof}
We also would like to point out that for the cop catching number one can obtain a lower bound of one larger than the dimension of the manifold by a similar argument as in Theorem~\ref{thm:lowerbound}. Further, Lemma~\ref{lem:capturepolygon} can be extended to higher dimensional hyperbolic spaces, if $n$ cops guard the boundary of a convex polytope which contains the robber, then they can catch the robber in finite time.
\section{Concluding remarks}
\label{sec:conclude}
One of the main open question is Conjecture~\ref{conj:Mohar2} on the cop win number of surfaces. A first step towards it would be to lower the multiplicative constant for upper bounds that are linear in $g$, as was done for graphs of genus $g$. Note that in general the cop number of a compact metric space can be infinite~\cite{georgakopoulos2023compact}. It would be nice to find more properties of spaces that guarantee a finite (or constant) cop win number.

When considering compact surfaces of constant curvature, it is still open whether (and how) the cop catch number grows with the genus of the surface. It is easy to show that a linear number (depending on the genus) of cops is sufficient using Lemma~\ref{lem:capturepolygon}, by cutting the surface into smaller parts using isometric paths (similarly as was done in~\cite{Mo22}). We obtained that on natural constant curvature surfaces the cop number is bounded by a constant (irrespectively of the genus). A first step towards understanding the cop catch number on constant curvature surfaces would be to find the exact cop catch number for the surfaces $S(g),S'(g)$ and $N(g)$. More importantly, is the cop catch number bounded by a constant on hyperbolic surfaces?

\bibliographystyle{amsplain}
\bibliography{bifile.bib}

\appendix
\section{Proofs of technical lemmas}
\label{ap:proofs}

\begin{proof}[Proof of Lemma~\ref{lem:ball-sys}]
     Consider $B(a,r)$ as a disk in the universal covering space $\mathcal{D}$. Let $x,y \in B(a,r)$, then the isometric path $g$ between $x$ and $y$ in the copy of $B(a,r)$ in the universal covering space $\mathcal{D}$ is a locally isometric path from $x$ to $y$ in $S$. Suppose $\overline{xy}\neq g$. Then the closed path $P$ going from $x$ to $y$ along $\overline{xy}$ and from $y$ to $x$ along $g$ is of length less than $\sys(S)$ and consists of two paths which are locally isometric and bounds a region which is non-empty.

    The Gauss-Bonnet Theorem states that for a region $R\subset S$ with sectional curvature $K$,  piecewise smooth boundary $\partial(R)$, 
    \begin{align*}
        \int_R K \, dA+\int_{\partial(R)} k_g \, ds=2\pi \chi(R),
    \end{align*}
    where $dA$ is the element of area of the surface, and $ds$ is the line element along the boundary, and $\int_{\partial(R)} k_g ds$ is the sum of the corresponding integrals of the geodesic curvature $k_g$ along the smooth portions of the boundary, plus the sum of the angles by which the smooth portions turn at the corners of the boundary.
    
    Consider the curve $P$ which is piecewise geodesic. Since the curve $P$ has length strictly smaller than $\sys(S)$, it has to be contractible. The curve $P$ bounds a region $R$. Since the curvature of locally isometric curves is $0$, by Gauss-Bonnet 
    \begin{align*}
        \int_{R} -1 dA + \theta_1-\pi +\theta_2-\pi = 2\pi \chi(R),
    \end{align*}
    where $\theta_1,\theta_2$ are the exterior angles between the locally isometric paths at $x$ and $y$.
    Since $P$ is contractible, the Euler characteristic of $R$ is $1$, hence $\theta_1=\theta_2=2\pi$ and $\int_R 1 dA=0$, which means that the isometric path between $x$ and $y$ goes along the geodesic $g$.
\end{proof}

In the following proofs, denote $d(B,C)=a$, $d(A,C)=b$ and $d(A,B)=c$, $\alpha$ is the angle at $A$, $\beta$ is the angle at $B$ and $\gamma$ is the angle at $C$.

\begin{proof}[Proof of Lemma~\ref{lem:closertogether}]
This is a simple application of the Pythagorean Theorem for hyperbolic triangles. Suppose $X,Y\in \overline{AC}$ with $X\in \overline{YC}$. Since $\cosh$ is strictly increasing on the positive real axis and $\acosh$ is strictly increasing
\begin{align*}
d(Y,B)&=\acosh(\cosh(a)\cosh(d(Y,C)))\\
&\geq \acosh(\cosh(a)\cosh(d(X,C)))= d(X,B)
\end{align*}
with equality if and only if $X=Y$. The lemma follows by using this inequality twice, 
\begin{align*}
    d(X_2,B)-d(X_3,B) \leq c-d(X_3,B) \leq c-d(X_1,B).
\end{align*}
Clearly, $d(X_2,B)-d(X_3,B)=c-d(X_1,B)$ if and only if 
$X_2=A$ and $X_3=X_1$.
\end{proof}

\begin{proof}[Proof of Lemma~\ref{lem:shifttoright}]
    For a vertex $X\in \overline{AC}$, by the Pythagorean Theorem for hyperbolic triangles
    \begin{align*}
d(X,B)=\acosh(\cosh(a)\cosh(d(X,C))).
    \end{align*}
    The second derivative with respect to $d(X,C)$ is 
\begin{align*}
    \frac{\partial^2 d(X,B)}{\partial d(X,C)^2} =\frac{(\cosh(d(X,C)) \cosh(a) \sinh^2(a))}{(-1 + \cosh^2(a) \cosh^2(d(X,C)))^{3/2}}.
\end{align*}
Therefore, $\frac{\partial^2 d(X,B)}{\partial d(X,C)^2}>0$ for positive values of $a$ and $d(X,C)$, hence $d(X,B)$ is convex for $d(X,C)>0$.  This means that the secant 
\begin{align*}
    \frac{d(X,B)-d(Y,B)}{d(X,C)-d(Y,C)}
\end{align*}
is monotonely increasing (in both $d(X,C)$ and $d(Y,C)$). Consequently,
\begin{align*}
\frac{c-d(X_1,B)}{d(A,X_1)}&=
    \frac{c-d(X_1,B)}{b-d(X_1,C)}\\
    &\geq 
    \frac{d(X_2,B)-d(X_3,B)}{d(X_2,C)-d(X_3,C)}=\frac{d(X_2,B)-d(X_3,B)}{d(X_2,X_3)},
\end{align*}
and since $d(A,X_1)=d(X_2,X_3)$, this proves the lemma.
\end{proof}

\begin{proof}[Proof of Lemma~\ref{lem:intuition}]
By the law of cosines
\begin{align*}
    \gamma=\arccos\left(\frac{\cosh(c)-\cosh(b)\cosh(a)}{-\sinh(a)\sinh(b)}\right).
\end{align*}
Since $\arccos$ is monotonely decreasing, we only need to show that $f$ as defined by below equation is monotonely increasing. We define
\begin{align*}
f(b):=\frac{\cosh(c)-\cosh(b)\cosh(a)}{-\sinh(a)\sinh(b)}.
\end{align*}
The derivative of $f$ is 
\begin{align*}
f'(b)&=\frac{\cosh(a)\sinh(a)\sinh(b)^2+\sinh(a)\cosh(b)(\cosh(c)-\cosh(b)\cosh(a))}{(-\sinh(a)\sinh(b))^2}\\
&=\frac{-\cosh(a)\sinh(a)+\sinh(a)\cosh(b)\cosh(c)}{\sinh(a)^2\cosh(b)^2},
\end{align*}
where we use that $\sinh(b)^2-\cosh(b)^2=-1$. Now since $\alpha<\frac{\pi}{2}$ it follows that $\cosh(b)\cosh(c) \geq \cosh(a)$ and hence the derivative is non-negative. This proves the lemma.

\end{proof}

\begin{proof}[Proof of Lemma~\ref{lem:mean-value}]
    Note that by the mean value theorem for $x<y$ $$\acosh(w\cosh(x))-\acosh(w\cosh(y))\geq c \cdot (y-x)$$ where $c=\min_{\xi \in [x,y]} \partial/\partial \xi (\acosh(w \cosh(\xi)))$. Since $w>1$, the partial derivative is 
\begin{align*}
\frac{\partial}{\partial \xi}\acosh(w \cosh(\xi)) = \frac{w \sinh(\xi)}{\sqrt{w^2-1 + w^2 \sinh^2(\xi)}}.
\end{align*}
Taking $x,y \in [t_1, t_2]$ with $t_1>0$ and noting that $\sinh(x)$ is monotonely increasing we get that 
\begin{align*}
    c\geq \frac{w \sinh(t_1)}{\sqrt{w^2-1 + w^2 \sinh^2(t_2)}}\geq \eta
\end{align*}
for some $\eta>0$.
\end{proof}

\end{document}